\documentclass[preprint,12pt]{elsarticle}

\usepackage[margin=2.0cm]{geometry}

\usepackage{graphicx}
\usepackage{amssymb}
\usepackage{amsmath}
\usepackage{siunitx}
\usepackage[dvipsnames]{xcolor}
\usepackage{commath}
\usepackage{amsfonts}%
\usepackage{graphicx,verbatim}
\usepackage{tikz}
\usetikzlibrary{cd}

\usepackage[pagebackref=false,colorlinks=true, pdfstartview=FitV, linkcolor=violet,citecolor=red, urlcolor=blue]{hyperref}

\usepackage{amsthm}
\newtheorem{lthm}{Theorem}

\newtheorem{ldef}{Definition}

\usepackage{dirtytalk}

\usepackage{hyperref}

\newcommand{\cC}{\mathcal{C}}

\newcommand{\tA}{\widetilde{A}_3}
\newcommand{\tB}{\widetilde{B}_3}
\newcommand{\tC}{\widetilde{C}_3}
\newcommand{\tX}{\widetilde{X}_3}

\usepackage[font=small,labelfont=bf]{caption}

\newtheorem{theorem}{Theorem}
\newtheorem*{theorem*}{Theorem}
\newtheorem{definition}{Definition}
\newtheorem{proposition}{Proposition}

\newtheorem{example}{Example}

\newtheorem*{conjecture*}{Conjecture}
\newtheorem{convention}{Convention}
\newtheorem{question}{Question}
\newtheorem{remark}{Remark}
\usepackage{lineno}

\journal{ArXiv}

\begin{document}

\begin{frontmatter}


\title{Knots and Coxeter Groups}



\author{Dylan Burke}
\ead{dylan.burke@mail.utoronto.ca}

\author{Geoffrey Cuff-Chartrand}
\ead{geoffrey.cuffchartrand@mail.utoronto.ca}

\author{Malors Espinosa}
\ead{srolam.espinosalara@mail.utoronto.ca}

\author{Mateusz Kazimierczak}
\ead{mateusz.kazimierczak@mail.utoronto.ca}

\author{Mohammadamin Mobedi}
\ead{mobedimohammadamin@student.deanza.edu}


\begin{abstract}

In this paper we study knots created by galleries in the affine Coxeter complex of type $\tB$. We bound the stick number by $40$ and prove that the smallest length of threefold rotationally symmetric trefoils is $42$. We  construct explicit galleries that knot as $9_{35}, 9_{40}, 9_{41}$ and $9_{47}$ in a way that has threefold rotational symmetry. We explain the construction of these galleries for $9_{47}$ carefully. We conclude with three questions inspired by this work. 

\end{abstract}
\end{frontmatter}


\section{Motivation and Results}\label{sec: motivation and results}

This work is concerned with the construction of knots under certain restrictions. One can think of a knot as a piece of string whose endpoints are glued together and that does not self-intersect. In this context, the knot is built from the string but the string itself puts no further restrictions on the study of the knot. However, if instead of using string we glue sticks together, then the pieces with which the knot is built become important.

Mathematically, in the first context we consider (tame) knots as embeddings of the circle into $\mathbb{R}^3$. On the second context, we consider polygonal closed paths that do not self-intersect. A well known result states that to create a nontrivial knot with a polygonal curve requires at least $6$ straight lines (i.e. one needs at least six sticks).  This number is usually called the \textit{knotting number} or the \textit{minimal stick number}. However, one relaxes this to more general set of restrictions and also call them \textit{minimal stick number} (for those restrictions).

Our first result goes in this direction. We prove
\begin{lthm}\label{lthm: minimal is 42}
  In the Coxeter Tesselation $\tB$ the minimal stick number is at most $40$.  
\end{lthm}

This result is one in a long family of results of this kind. The classical result of Diao states that in the cubic lattice, where we require that contiguous sticks glue at centers of adjacent cubes, has minimal stick number equal to $24$ (see \cite{Diao}). Some other minimal stick numbers, for other type of translational tessellations, are given in \cite[Theorem 2.1]{mannminimalknottingnumbers}. There are other type of conditions one can put besides where the endpoints can be. For example, if all sticks must have the same length the minimal stick number is $6$ (see \cite{Equilateralknots} and the images in \cite{Equilateralknotsimages}). 

The Coxeter tessellation $\tB$ is a refinement of the standard cubic tessellation of Euclidean space. One further divides each cube into $24$ pyramids by joining the center to each of the vertices and each center of the faces. We require that the polygonal paths have endpoints at contiguous centers of these pyramids. We proved theorem \ref{lthm: minimal is 42} by an exhaustive process aided by (simple house) computers. More concretely, we produced the list of all possible closed paths until we found a knotted one. We will explain this in section \ref{sec: knotting length in the Coxeter Lattices} below.

We did not pick this tessellation randomly. If we color the four vertices of each pyramid with different colors with the restriction that vertices that are opposite to the same face must have the same color then the group of symmetries of this colored tessellation is an affine Coxeter group of type $\tB$. In particular, paths within pyramids can be codified by words in the generators of this group. Thus, instead of codifying sets of vectors that represent consecutive vertices, we wrote words in the generators and studied them with an interactive program that we coded. 
We will review Coxeter groups in section \ref{sec: Irreducible Coxeter Complexes of rank 4} below.

In order to motivate the following set of results, we explain what lead us to them. The lists that we had to produce of words of increasing length to be checked became very large. For the longest time we could not produce a trefoil. Thus, we decided to \textit{convert} a cubic trefoil into pyramids. We chose one of minimal length (i.e. $24$ cubes). However, we picked a cubic trefoil that has threefold rotational symmetry, that is, it is built of three copies of the same piece. Each piece is the previous one rotated by $120$ degrees around certain axis. This led to a path of pyramids that also has threefold rotational symmetry and length $48$. 

However, in $\tB$, this rotation is a symmetry of the colored tessellation and thus it is a word in the generators. This word has dual purposes: on the one hand it codifies the \say{piece} to be repeated and, on the other hand, it is an element of the Coxeter group of order $3$. Hence, instead of looking for just increasingly long words, we searched for words of order three and of length at most $14$. This was on the limit of what we could handle in maneagable time with our computers. This motivated us to prove
\begin{lthm}\label{lthm: order 3 iff order 3}
   A knot $K$ has threefold rotational symmetry if and only if it can be represented by a word of order $3$ in $\tB$. 
\end{lthm}

Of course, one now has to address the potential issue that the smallest knotting number might not be achieved by a knot with threefold symmetry. For example, in \cite[Theorem 2.1]{mannminimalknottingnumbers} it is proven that the minimum sticking number of the simple hexagonal lattice (sh) is $20$. This lattice admits rotational threefold symmetry and there are trefoils with threefold symmetry in it. However, the minimum cannot be achieved by such a trefoil because $20$ is not a multiple of $3$ and thus it cannot be built out of three equal pieces. 

Another example concerns the cubic lattice. We have mentioned the minimal stick number for it is $24$. Instead, one can require that the endpoints of sticks are in the lattice points and edges parallel to the axis. The minimal stick number of the latter, sometimes referred as the \textit{lattice} stick number, is $12$ (see \cite{huhohlatticestick}). We thus see that both possibilities happen.

As we have mentioned, Diao proved the former result and gave an example of a minimal trefoil (see \cite[Figure 1]{Diao}). It consists of $24$ contiguous unit sticks, which in turn form $13$ different \say{bars}. Consecutive bars are orthogonal and all of them are parallel to the axis. Thus, this is not a symmetric trefoil nor an example of a minimal one for the lattice stick number. However, in \cite[Figure 1 (a)]{huhohlatticestick} another trefoil is presented. This is rotationally symmetric, built out of three equal pieces. Moreover, it is a minima for both the stick number and the lattice stick number.

We know that as the number of sticks increases, without further restriction, the families of knots separate into equivalence classes that are finer than those of usual knot equivalence (see \cite{calvogeometricknot}). When we put conditions several equivalence classes disappear because they do not have representatives that satisfy those conditions. They get sieved away. Thus, the minima of stick numbers with respect to these conditions redistribute. Components where the minima lie, don't necessarily have symmetric representatives. On other cases, minima are of both types, symmetric or non-symmetric. For instance, it easy to see that the minima shown in \cite{Diao} and \cite{huhohlatticestick} lie in the same component. One might wonder about some restrictions such that certain components, of knots built out of the same number of sticks,  have symmetric minima while others don't. We do not have an example of one.

The above discussion is given to motivate making the \say{symmetric pieces and their construction} part of the object of study, as opposed to something that happens as a consequence of simplifying searches. We reached to this need of symmetry also to simplify our search but it made us focus strongly on the construction of the pieces themselves.

This motivates us to give the following definitions in section \ref{sec: knotting length in the Coxeter Lattices}.
\begin{ldef}
Let $K$ be a knot that admits threefold rotational symmetry. The minimal number such that there exists a threefold rotational gallery in $\tB$, of such length, that knots as $K$ is called the \textbf{$3$-minimal stick number} of $K$ in $\tB$. We will denote it by $l_3(K)$ or $l_{3, \tB}(K)$.

The minimal number such that there exists a rotationally symmetric gallery that knots in nontrivial way is called the $3$-\textbf{minimal stick number} of $\tB$. We denote it by $l_3(\tB)$.

\end{ldef}
We thus proved
\begin{lthm}\label{lthm: Sym minimal is 8}
    The $3-$rotational stick number of $\tB$ is $42$ and is achieved by a symmetric trefoil. Thus $l_3(3_1) = 14$.
\end{lthm}
In particular, we realize that $\tB$ does not have threefold symmetric minimums. One can define analogous numbers for other symmetries or other lattices. For example, for the cubic lattice or for the (sh) one we denote them by $l_3(C)$ and $l_3(\mbox{sh})$. The results above show
\begin{equation*}
    l_3(C) = 8,
\end{equation*}
because there are symmetric minimums. For (sh) we will justify in section \ref{sec: questions}
\begin{equation*}
    8\le l_3(\mbox{sh})\le 12.
\end{equation*}
We do not know what the exact number is. 

With this in mind, we shift our attention to bounding $l_3(K)$ for the next prime knots that have threefold symmetry. These are $9_{35}, 9_{40}, 9_{41}$ and $9_{47}$. We cannot expect to find by exhaustive search words that represent their \say{piece} to be repeated. Instead we will construct a cubic piece that, when rotated $120$ and $240$ degrees around certain axis, produces the sought knot. 

The method to proceed is as follows. We first find the cubic piece that gives this knot in the cubic lattice and then we convert it to $\tB$. We will discuss this translation process in section \ref{sec: translation from the cubes to the pyramids}. However, because the cubic lattice is not a Coxeter lattice, the codification of paths has to be done with care to have an appropriate substitute of \say{order 3 words}.

Thus, the problem gets reduced to finding the cubic pieces for knots that have threefold symmetry. As is well known, two problems can occur. On the one hand a knot diagram might not show the symmetries that the physical knot might admit or, on the other hand, a knot diagram might show the symmetry, but it is not easily translatable to a construction of this knot that has this symmetry. 

We noticed that in all the cases we tried we could do an inductive constructive process. It begins with analyzing a knot diagram that has threefold symmetry. Given such a knot diagram, we aim at constructing the \say{cubic} pieces of progressively smaller regions. These regions are either repeated three times, due to the rotational symmetry, or surround the center of rotation. Ultimately, we have to be able to construct the smallest pieces that do not surround any other smaller one. In all the cases we tried, this is possible. 

The question of whether a knot with threefold symmetry (or more general rotational symmetries) have a diagram that shows it is well studied (see \cite{COSTA} and the references therein). Indeed it is proven that such a diagram exists for alternating knots with rotational symmetries. These diagrams are not necessarily compatible with given lattices or similar restraints. It is interesting to know if one can find such diagrams that are always adaptable to constructing the pieces from them.

For example, for $12_{503}$, we see in \cite[Example 5, Figure 13(b)]{COSTA} a diagram with threefold symmetry. Using it we design a piece
\begin{center}
    \includegraphics[scale = 0.3]{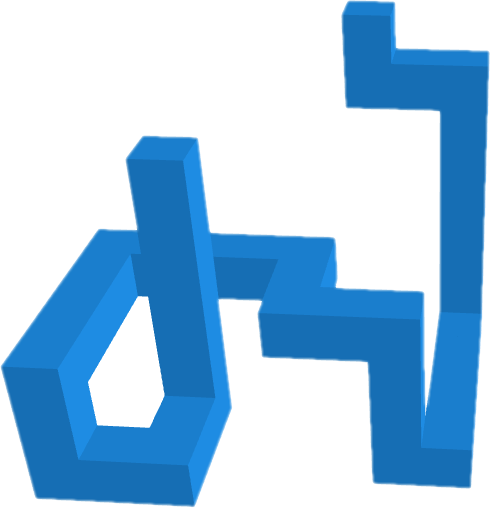}
    \captionof{figure}{The piece for $12_{503}.$}
    \label{fig: The piece for 12_{503}.}
\end{center}
Using three copies of this figure we construct $12_{503}$ in the cubic lattice and with threefold rotational symmetry. 
\begin{center}
    \includegraphics[scale = 0.3]{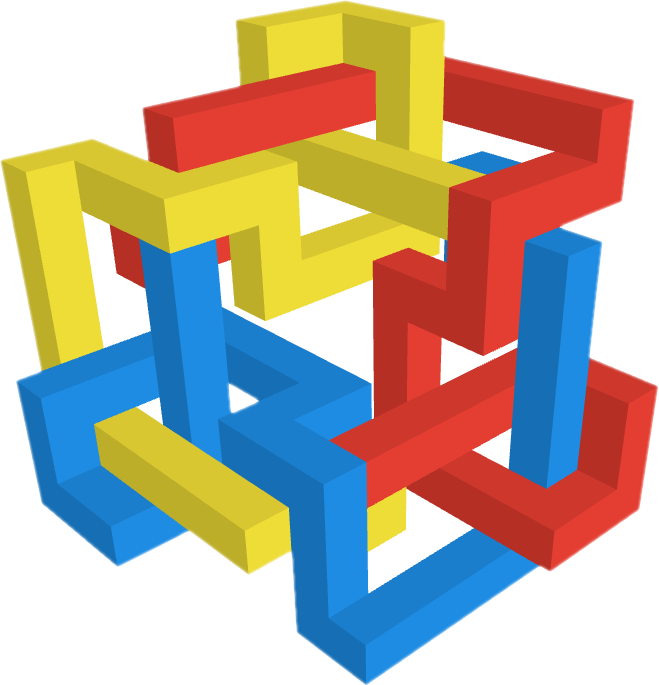}
    \captionof{figure}{The knot $12_{503}$ built out of three pieces rotated by $120$ degrees}
    \label{fig: The knot 12_{503} built out of three pieces rotated by 120 degrees}
\end{center}
Even though we are showing the piece as a single connected path of cubes, its constructions is by parts. We append the required prolongations, to preserve the rotational symmetry, to the central figure we referred above. 

In section \ref{sec: Order three galleries} we will explain this carefully for $9_{47}$. We also give the analogous of figures \ref{fig: The piece for 12_{503}.} and \ref{fig: The knot 12_{503} built out of three pieces rotated by 120 degrees} above for $9_{35}, 9_{40}, 9_{41}$. This leads to Theorem \ref{thm: words for order three for cross 9} and its proof, which give the cubic words and Coxeter ones for these knots, respectively. In this way we get strict bound for $l_3(9_i)$ for $i = 35, 40, 41, 47$. We do not know the exact values.

For $9_{47}$ it is known that its stick number is $50$ (see \cite[Table 1]{rensburgjanseminimalknots}). Thus, it will not have a symmetric minimum in the cubic lattice. The piece we produce has $46$ cubes in it. Thus
\begin{equation*}
    18\le l_{3, \mathbb{Z}^3}(9_{47}) < 46.
\end{equation*}

Finally, in section \ref{sec: questions} we put some questions that were motivated by this work and that the authors believe are interesting to study.

The knot diagrams of $9_{35}, 9_{40}, 9_{41}$ and $9_{47}$, as well as an image of a rhombic dodecahedra below, were taken from their wikipedia articles. The cubic galleries and pieces, as those in figures \ref{fig: The knot 12_{503} built out of three pieces rotated by 120 degrees} and \ref{fig: The piece for 12_{503}.} and similar, were produced with Voxel Paint app. The rest of the figures were created with our aforementioned program or drawn by the authors.

\section{Irreducible Coxeter Complexes of rank $4$}\label{sec: Irreducible Coxeter Complexes of rank 4}

There are three irreducible affine Coxeter groups of rank $4$. They are denoted by $\tA, \tB$ and $\tC$. Each of them is associated to a tesselation into pyramids of $\mathbb{R}^3$. We will describe them here and explain how we use them. However, that the results we are using hold won't be proved for the sake of space. We refer the reader to \cite[Chapter 2, 3]{BUILDINGS} for proofs of all of these facts.

We encourage the reader familiar with the theory of Coxeter groups to move on onto section $3$. We highlight here for their sake that the above groups have the presentations
\begin{small}
    \begin{align*}
    \tA &= \langle A, B, C, D \mid A^2 = B^2 = C^2 = D^2 =(AB)^4 = (AC)^2 = (AD)^2 =(BC)^3 = (BD)^2 = (CD)^4 = 1 \rangle,\\
    \tB &= \langle A, B, C, D \mid A^2 = B^2 = C^2 = D^2 = (AB)^2 = (AC)^3 = (AD)^3 =(BC)^3 = (BD)^2 = (CD)^4 = 1 \rangle,\\
    \tC &= \langle A, B, C, D \mid A^2 = B^2 = C^2 = D^2 =(AB)^2 = (AC)^3= (AD)^3 =(BC)^3 = (BD)^3 = (CD)^2 = 1 \rangle.
\end{align*}
\end{small}

Finally, we follow the next
\begin{convention}
    When we describe general statements that apply to the three constructions above we will write $\tX$.
\end{convention}

\subsection{Construction of the tesselation and galleries}\label{subsec: construction of tess and gal}

Each of the aforementioned tesselation is constructed in similar way. We begin with a tesselation of $\mathbb{R}^3$ by larger polyhedra that, in turn, get divided into these pyramids. To each one of the vertices of these pyramids we give a tag, frequently called \textit{color} or \textit{type}, in such a way that vertices that are opposite to a given face have the same color.  We now describe each one separately starting from the simpler one.

\begin{description}
    \item[$\tB$:] We begin with the standard tesselation of cubes of euclidean three dimensional space. We divide each one of these cubes into $24$ pyramids as follows: we join the center of the cube with each one of the eight vertices and each of the six centers of each face. These are the edges of the $24$ pyramids. We do this on each cube and obtain the $\tB$ tesselation.

    The \say{coloring} goes as follows: $D$ is the center of the cube. $C$ is the color of the center of the faces. Finally, $A$ and $B$ are the colors of the vertices in alternating fashion.
    
    \item[$\tC$:] This one also begins with the tesselation of cubes of euclidean space. We divide each one of these cubes into $48$ pyramids as follows: we join the center of the cube with each one of the eight vertices, each of the six centers of each face and each of the twelve middle points of the sides. These are the edges of the $48$ pyramids. Each pyramid has as vertices the center of the cube, a middle point of an edge, a center of a face and a vertex. We do this on each cube and obtain the $\tC$ tesselation.

    The \say{coloring} goes as follows: $D$ is the center of the cube. $C$ is the color of the center of the faces. $B$ is the color of the midpoints of the edges. Finally, $A$ is the color of the vertices.

    \item[$\tA$:] This one begins with a tesselation of the euclidean space into \textit{rhombic dodecahedra.} The image of such a figure is as follows
    \begin{center}
    \includegraphics[scale = 0.15]{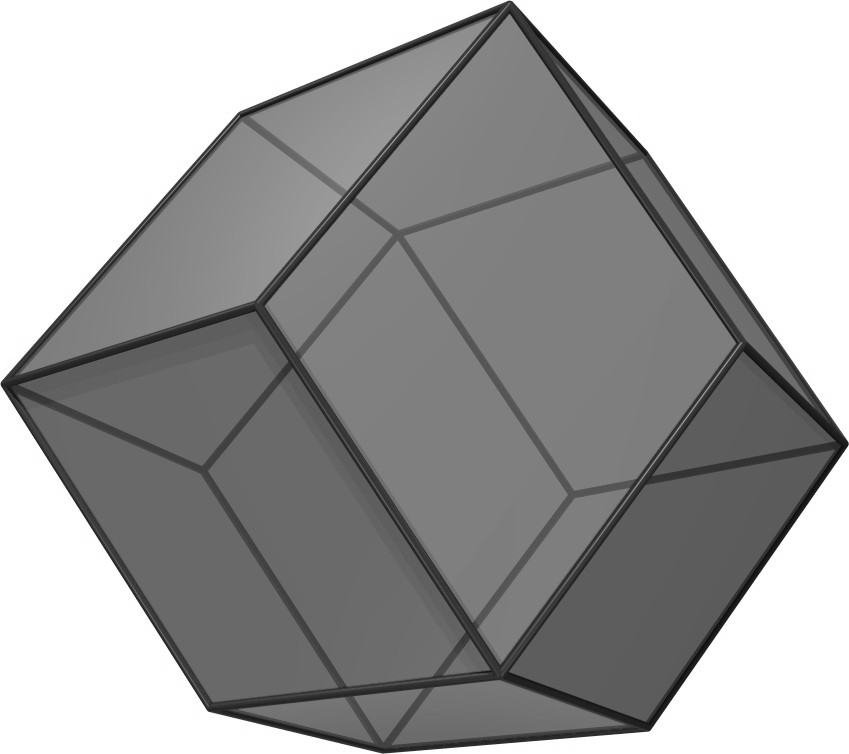}
    \captionof{figure}{An image of a rhombic dodecahedron taken from the wikipedia article}
\end{center}
    To visualize the construction, imagine one of the cubes of the classical tesselation and glue to each one of the six sides the square pyramid obtained by joining the vertices of a face to the center of the adjacent cube along this face. These are polyhedra of $12$ sides, because faces adjacent along a face of the cube are coplanar.
    
    We now join the center of the cube with each one of the $24$ vertices, each of the $12$ centers of each face and each of the twelve middle points of the sides. These are the edges of the $48$ pyramids. Each pyramid has as vertices the center of the cube, a middle point of an edge, a center of a face and a vertex. We do this on each cube and obtain the $\tC$ tesselation.

    The \say{coloring} goes as follows: $D$ is the center of the rhombic dodecahedra. $C$ is the color of the vertices that are also centers of cubes in the cubic lattice (i.e. the extra corners of the pyramids). Finally $A$ and $B$ are the vertices remaining in alternating fashion (i.e vertices that were vertices of the original cube).
\end{description}

We emphasize that $D$ is \textit{always} the center of the \textit{original} polyhedron being subdivided. An important difference is that in $\tA$ and $\tC$ all the polyhedra with the coloring that comes from the above descriptions are translations of each other. However, in $\tB$ this is not the case: there are different colorations, related by one being a reflection of the other, that appear in alternating pattern.

The following is the fundamental definition in which everything will be based.

\begin{definition}\label{def: gallery}
    A  \textbf{gallery} in $\tX$ is a path of pyramids such that consecutive pyramids share a face.
\end{definition}

The following is an important proposition
\begin{proposition}\label{prop: noone fix C}
    The only congruent transformation that can fix a pyramid of $\tX$ is the identity.
\end{proposition}
\begin{proof}
    To see this put the origin in one of the vertices and the three edges are linearly independent. Thus, the only transformation that fixes them is the identity. Notice we used the coloration to claim each of the vectors is fixed, otherwise we could swap \textit{similar} vertices.
\end{proof}

Now we progressively explain how we codify these galleries and the geometry associated to them. To describe a gallery we need two pieces of information: (1) the starting pyramid and (2) the sequence of faces we are crossing. If we suppose we are given the starting pyramid, which we call $\cC$ (in the language of Coxeter complexes and Buildings, these are called \textit{chambers}, hence the $\cC$), then we codify the path as follows: \textit{we write from \textbf{left to right} the color of the vertex \textbf{behind} the face we are crossing.}

Given our initial pyramid, the four adjacent pyramids are associated to each one of the vertices $A, B, C$ and $D$. However, more is true: each of these faces is contained in a plane and reflecting across any one of these planes \textbf{leaves the \textit{colored} tesselation} invariant. In other words, reflection on the planes of the faces of any given pyramid is a \textbf{symmetry} of the colored tesselation. We will from now on just say tesselation to mean \say{colored tesselation} unless otherwise stated. These symmetries, with respect to $\mathcal{C}$, are also denoted by $A, B, C$ and $D$. We emphasize that the reflections we are denoting by these letters depend on the initial pyramid.

The following result is what makes this construction very useful.

\begin{theorem}
    Let $\cC$ be any initial pyramid of the teselation corresponding to $\tX$ and $A, B, C, D$ the reflections described above. Then every symmetry (i.e. congruent transformation of euclidean space) that leaves the tesselation invariant is a composition of $A, B, C$ and $D$.
\end{theorem}
\begin{proof}
    See \cite[Lemma 3.31, Proposition 3.32]{BUILDINGS}.
\end{proof}

Using proposition \ref{prop: noone fix C}, we conclude that two symmetries $\Phi_1$ and $\Phi_2$ with
\begin{equation*}
    \Phi_1(\cC) = \Phi_2(\cC)
\end{equation*}
must coincide. Indeed, $\Phi_2^{-1}\circ \Phi_1$ fixes $\cC$ and thus must be the identity. As a consequence, we get
\begin{proposition}\label{prop: symmetries oneone chambers}
    The symmetries of $\tX$ are in bijective correspondence with the pyramids of the tesselation.
\end{proposition}
\begin{proof}
    Indeed, we just saw injectivity. However, it is clear that any two pyramids are related by a congruent transformation.
\end{proof}

Given the above result, we find ourselves in the following possible inconsistency. Suppose we have a gallery, starting at $\cC$, say $g = (s_1,..., s_k)$ where $s_i \in\{A, B, C, D\}$. Then
\begin{equation*}
    \Phi = s_1\circ\cdots \circ s_k,
\end{equation*}
is a symmetry of $\tX$ and thus $\Phi(\cC)$ is a pyramid of the tesselation. We want this to be the ending pyramid of the gallery! The reason why the codification we explained above is so useful is because this indeed functions. We have

\begin{proposition}\label{prop: coherent codification}
    In the context of the above discussion, $\Phi(\cC)$ is the ending pyramid of the gallery codified by $g = (s_1,..., s_k)$ that starts at $\cC$. In particular, galleries that begin and end in the same place codify the same symmetry. 
\end{proposition}
\begin{proof}
    See \cite[Proposition 3.24]{BUILDINGS}. There they do it for minimal galleries with elementary homotopies. It is easy to extend it to nonminimal ones using general homotopies.
\end{proof}

This leads to six particular galleries, each obtained by alternating two given letters (or similarly, crossing faces that rotate around an edge of a pyramid). These galleries return to the starting chamber. For example, for $\tB$ these galleries are
\begin{align*}
    &\;(A, D, A, D),  
    (B,D,B,D),
    (A,B,A,B)\\
    &\;(A,C,A,C,A,C), 
    (B,C,B,C,B,C),
    (C,D,C,D,C,D,C,D) 
\end{align*}
In particular, because the galleries close, we have that the symmetry given by performing one jump across a face and then jumping across a different one have finite order and can be written as
\begin{equation*}
    (XY)^{m_{XY}} = 1, X, Y\in\{A, B, C, D\}, X\neq Y.
\end{equation*}
Of course, we can admit $X = Y$ above and get $X^2 = 1$. With this construction we finally reach the following fundamental result.

\begin{theorem}\label{thm: coxeter presentation}
    The group of symmetries of $\tX$ admits the following \textbf{presentation}
    \begin{equation*}
        \langle A, B, C, D \mid (XY)^{m_{XY} = 1}, X, Y\in\{A, B, C, D\}\rangle
    \end{equation*}
    where $m_{XX} = 1$. 
\end{theorem}
\begin{proof}
    See \cite[Subsection 3.4.2, Lemma 3.72]{BUILDINGS}.
\end{proof}
In general, groups that admit these kind of presentations are called Coxeter Groups. All of them admit a tesellation of a complex, called the Coxeter Complex, and where the above results generalize. 

We will not prove the above result. However, we must mention the main technical result of its the proof because we will use it. 

\begin{proposition}
    Two galleries in $\tX$ that begin and end in the same pyramids can be obtained one from the other by performing a sequence of the following rules:
    \begin{description}
        \item[T1:] The addition or removal of two consecutive $X, X$ in any part of the gallery.

        \item[T2:] The exchange of $(X, Y, ...)$ for  $(Y, X...)$, with $X\neq Y$, and with the tuples above of size $m_{XY}$.
    \end{description}
    These rules $T1$ and $T2$ are called \textbf{Tits M-moves}.
\end{proposition}
\begin{proof}
See \cite[Theorem 2.33]{BUILDINGS}. There they discuss the word problem and reduced words. We have framed it above terms of galleries.
\end{proof}

\subsection{Finite orders}

We will be concerned with finite order symmetries and thus we list them for $\tB$. Similar propositions can be written down for $\tA$ and $\tC$. 

In \cite[Chapter 3, page 37]{coxeter} it is proven that any congruent transformation of the plane is a translation, a rotation, a reflection, a glide-reflection, a screw-displacement or a rotary-reflection. We recall that a \textit{glide reflection} is the compositions of a reflection and a translation where the translation vector lies in the plane of the reflection. A \textit{screw-displacement} is a rotation composed with a translation where the axis of rotation is collinear with the translation vector. Finally, a \textit{rotary-reflection} is a rotation composed with a reflection where the axis of rotation is perpendicular to the plane of reflection.

\begin{proposition}
    The possible finite orders for a symmetry of $\tB$ are $1, 2, 3, 4$, or $6$. They are achieved as follows: reflections have order 2; there are rotations of orders $2$, $3$ and $4$; the rotary-reflections have order $2, 4$ and $6$. 
\end{proposition}

\begin{proof}

     
If the transformation corresponding to the gallery is a translation, glide-reflection, and screw-displacement, the gallery has infinite order. Repeating the transformation will never bring the pyramids back to where they started as they will continue to move in the same direction.

If the transformation corresponding to the gallery is a reflection, the gallery has order two as all reflections have order 2.

If it is a rotation, we use the property that the transformation must leave the tesselation invariant and that the pyramids have no axis of symmetry. The axis of rotation must lie only on the edges of pyramids. Indeed, assume the axis of rotation passes through a pyramid. Then the points lying on the axis of rotation will be invariant under the transformation. Hence part of the pyramid the axis passes through maps to itself. However, the entire pyramid cannot map to itself as the pyramids have no axis of symmetry. Hence the pyramid the axis passes through does not map entirely to any pyramid in the tesselation. This is a contradiction.

As the axis of rotation cannot pass through a pyramid, it must lie along the edges or faces of the pyramids. The axis of rotation cannot lie along the face of a pyramid because of a similar argument as the one above. 

As the axis of rotation lies on the edge of a pyramid, it will pass through two vertices and therefore leaves these vertices invariant under rotation. Therefore the pyramids which share these two vertices will map to each other. Hence the possible orders of a rotation are the possible orders of the subgroups of the Coxeter group containing 2 elements

Rotary-reflections have multiple possibilities. The point of intersection of the axis of rotation and plane of reflection must lie on the vertex of a pyramid. The point of intersection is invariant under the transformation. If this point lies within a pyramid, that point of the pyramid will remain unchanged while the rest of the pyramid moves. As the pyramids have no symmetry, the pyramid will not map to itself and will also not map to any other pyramid as the point of intersection, which lies within the pyramid, is invariant.
If the point lies on a face, that point on the face will be invariant, and therefore the entire face will need to map to itself. However,as the face has no symmetry, this is not possible to do. Similarly, if the point of intersection lies on an edge, it must map the edge to itself which is not possible as only one point is left invariant.

The vertex the axis of rotation and plane of reflection intersect at will remain invariant, hence all the pyramids which share that vertex will remain invariant. Hence these pyramids must map to each other, and hence the order of the gallery must be achievable using a subgroup of reflections which result in the pyramids sharing a vertex, which is any subgroup that has 3 elements.

Therefore the possible orders of a Coxeter group are the same as the possible orders in the subgroups of the Coxeter group generated by 3 generators. These subgroups are the groups of symmetries of the links of the vertices, which are finite. Therefore the possible orders of these subgroups can be determined by checking each element.

\end{proof}


\section{Knots and Symmetries}\label{sec: knots and symmetries}

As we have discussed above, knots can have different types of symmetries. Some of those are of the kind achieved as symmetries of the tessellations $\tA, \tB$ and $\tC$ as we have discussed above. We are particularly interested in threefold rotational symmetries and we will restrict to those. Because we will work primarily with $\tB$ we will state our results for this tesselation. Many of these results can be extended to other ones but we leave those cases to the interested reader

\begin{proposition}\label{prop: knotrotational}
    A knot $K$ has $3-$rotational symmetry if and only if there exists a gallery
    \begin{equation*}
        (s_1, s_2,..., s_k)
    \end{equation*}
    in $\tB$ with $w = s_1s_2...s_k$ of order $3$ and such that 
    the polygonal curve created by joining consecutive centers
    of
    \begin{equation*}
        (s_1, s_2,..., s_k, s_1, s_2,..., s_k, s_1, s_2,..., s_k)
    \end{equation*}
    is a knot of the same type as $K$. In this situation, we will say that \textit{the knot $K$ is represented by a word of order $3$} or that $K$ is \textit{$3-$periodic}. 
\end{proposition}
\begin{proof}
The line that goes through opposite vertices of a cube is an axis of a threefold rotational symmetry of the cube. Thus, it is actually an axis of a threefold rotation of the whole cubic lattice. If we divide the cube into eight equal cubes, with side lengths half that of the original one, this axis is still an axis of a threefold symmetry of the whole cubic lattice created by subdivisions. Notice that these threefold rotations are also symmetries of $\tB$ if we use these cubes to create the tesselation of pyramids. The above remarks serve the purpose of allowing us to change the scale of the pyramids without changing the axis of rotation. 

Let us suppose now that $K$ has threefold symmetry. By aligning the axis of symmetries of $\tB$ with that of $K$, and letting the scale of the pyramids go to $0$, we can assume the following:
\begin{enumerate}
    \item For every simplex $P$ in the tesselation, $P\cap K$ is connected (possibly empty). The simplex are the vertices, the segments, the faces and the pyramids themselves.
    \item For every pyramyd $P$ in the tesselation, $P\cap K$ is untangled.
\end{enumerate}
We must now guarantee that without loss of generality we can avoid certain pathologies. The first one if the knot goes through one of the vertices of the tesselation. If this is the case, very close to the vertex, we cut the knot around the vertex and replace the string with a circular arch that avoids the vertex and all the edges emanating from it.
Because of the rotational symmetry, this happens in three different vertices, and in each we repeat this process in rotationally symmetric way. Furthermore, this does not change the knot type, nor creates new tangles. Notice that this path can be chosen in such a way to enter and leave each pyramid around the vertex at most once.

\begin{center}
    \includegraphics[scale = 0.6]{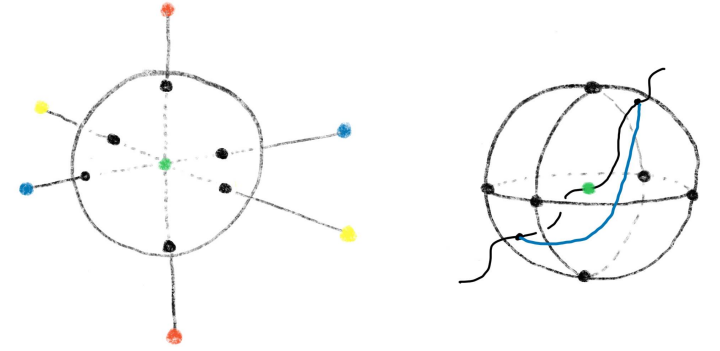}
    \captionof{figure}{How a sphere looks close to a (green) vertex and the substituting path on the surface.}
\end{center}

We claim this preserves the connectedness of $P\cap K$. If the vertex constituted this intersection, then now it is empty and there is nothing else to do. If not, the there is some point $q$ that remains in $P\cap K$ but that was not substituted when we replaced the path for a new one (maybe at the expense of making this substitution closer to the vertex). Thus, every point in the \textit{new} connected component, of some simplex, can be joined with this $q$ without leaving the pyramid, because once it leaves it, it does not come back to it in the new string. Thus, if there was another component, it was already there before pushing the vertex, which cannot be the case. Thus, without loss of generality, we can assume the knot does not go through the vertices of the tesselation.

Doing a similar argument, but now using a small cylinder around the closed segment on an edge that intetsects the knot, we can also avoid the edges. Notice that it might be that the knot intersects an edge at a single point in which case the cylinder is a disk. We leave the details to the reader.

Once we avoid vertices and edges all together, we do as follows. Pick any pyramid $P$ such that $P\cap K$ is nonempty. $K$ intersects $P$ for the first and for the last time in some points $P_1$ and $P_2$ in the interior of two faces $F_1$ and $F_2$. If $C_1, C_2$ and $C$ are the centers of $F_1, F_2$ and the pyramid, we have that the path that constitutes $K\cap P$ can be substituted by $P_1C_1CC_2P_2$. Doing this in a single pyramid, doesn't change the knot type. Notice it is for this statement that we used the second assumption given above. Furthermore, we can do it in the three rotationally symmetric pyramids corresponding to $P$. 

Thus, repeating this until we have exchanged all paths by straight paths between centers of faces and pyramids, we get indeed a path created by joining centers of pyramids that has threefolds symmetry That is, the knot is represented by a word of order three, as desired.
\end{proof}
\begin{remark}
    This proof works for other symmetries, but the type of symmetry represented by the gallery does not get characterized by its order unless it is the threefold rotation. 
\end{remark}

\section{Knotting Length in the Coxeter Lattices}\label{sec: knotting length in the Coxeter Lattices}

We will now discuss the minimal stick numbers. 

\begin{definition}\label{def: stick numbers}
Let $K$ be a knot. The minimal number such that there exists a gallery in $\tB$ that knots as $K$ is called the \textbf{minimal stick number} of $K$ in $\tB$. We will denote it by $l(K)$ or $l_{\tB}(K)$.
\end{definition}
It is clear that for any knot type there are galleries that knot in that way. In particular, $l(K)$ exists for every $K$. Thus there is a smallest number of sticks necessary to knot nontrivially.
\begin{definition}
    The minimal number such that there exists a gallery that knots in nontrivial way is called the \textbf{minimal stick number} of $\tB$. We denote it by $l(\tB)$. Thus
\begin{equation*}
    l(\tB) = \min\{l_{\tB}(K) \mid K\neq 0_1\}.
\end{equation*}
\end{definition}
\begin{remark}
    We will denote the analogous stick numbers in similar way. For example, for the cubic lattice $C$ we write $l_C(K)$ and $l(C)$.
\end{remark}

Threefold symmetric knots do not only have a minimal stick number but also a minimal number of sticks required to build a gallery that is itself threefold rotationally symmetric. We thus define
\begin{definition}\label{def: stick numbers}
Let $K$ be a knot that admits threefold rotational symmetry. The minimal number such that there exists a threefold rotational gallery in $\tB$, of such length, that knots as $K$ is called the \textbf{$3$-minimal stick number} of $K$ in $\tB$. We will denote it by $l_3(K)$ or $l_{3, \tB}(K)$.
\end{definition}
We have proved in proposition \ref{prop: knotrotational} that for any symmetric knot, we can construct galleries that are rotationally symmetric from words of order $3$ in $\tB$. This implies that $l_{3}(K)$ exists for every such $K$. In particular, there is a smallest number of sticks necessary to knot nontrivially and symmetricaly.
\begin{definition}
    The minimal number such that there exists a rotationally symmetric gallery that knots in nontrivial way is called the $3$-\textbf{minimal stick number} of $\tB$. We denote it by $l_3(\tB)$.
\end{definition}
The cubic lattice $C$ admits threefold rotational symmetries, thus we can define the analogous stick numbers. We will denote them by $l_{3, C}(K)$ and $l_3(C)$.

\begin{example}
    The trivial knot has $l(0_1) = 4$, corresponding for example to the word $ADAD$. It has $l_3(0_1) = 6$ corresponding to $CDCDCD$. Notice that these two are part of the defining relations of the Coxeter group $\tB$. 

    A trivial gallery that is constructed by a rotary-reflection of order $6$ is the following
    \begin{center}
    \includegraphics[scale = 0.5]{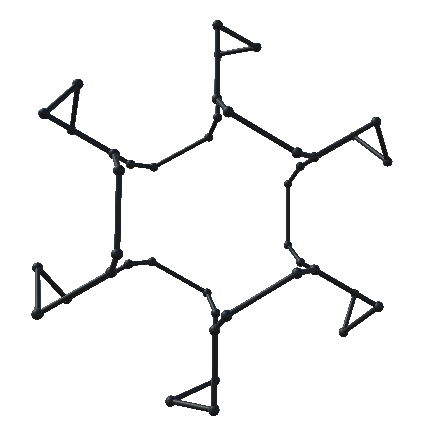}
    \captionof{figure}{$CABDACB$}
    \end{center}
    We can detect the reflection by noticing how the hinge that joins the \say{flaps} alternates sides.
\end{example}

\begin{example}
A symmetric trefoil of order $42$ is the next one.
\begin{center}
    \includegraphics[scale = 0.5]{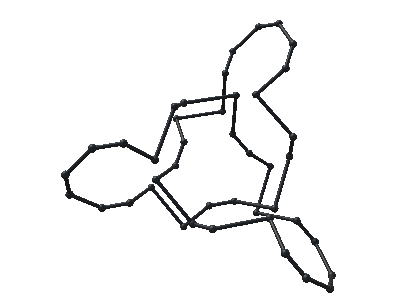}
    \captionof{figure}{$CBDCDCDCBADCBA$}
    \label{fig: minimal trefoil with threefold symmetry}
    \end{center}
\end{example}
As the next theorem will explain, this trefoil is a minima for the $3$-stick length.
\begin{example}
    The eight knot (\textit{i.e.}$4_1$) does not have threefold symmetry (see for instance \cite[Example 4.1]{BoyleOddOrder}). However it does have a rotary-reflection symmetry of order $4$.
\begin{center}
    \includegraphics[scale = 0.3]{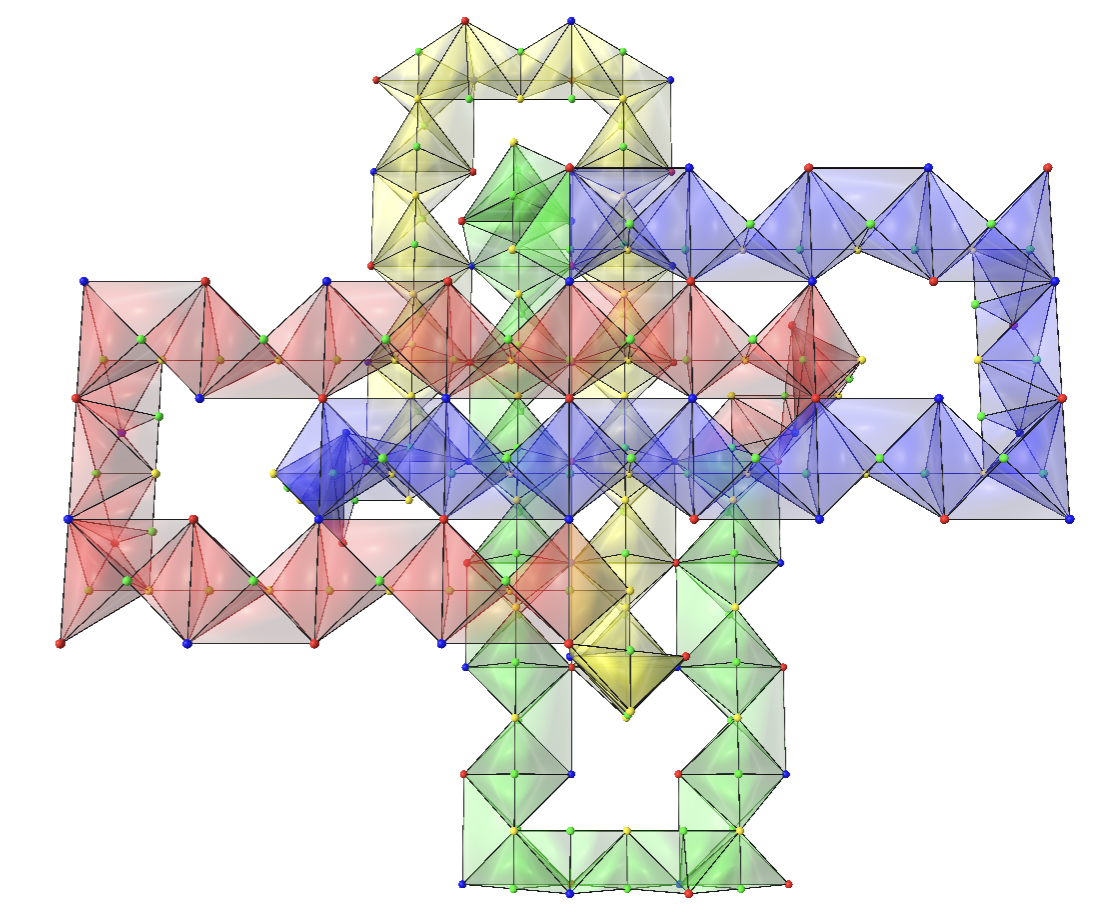}
    \captionof{figure}{An eight knot made by repeating four times a gallery of order $69$.}
    \end{center}  
The word being repeated is
\begin{align*}
   \;& DCBACDCBACDCBACDCBACACD\\
    & CABCDCBCABCDCABCDCABCDC\\
    & ABCDCABCDCABCDABCDCADCB.
\end{align*}
Notice that this is the composition of $69$ reflections, thus it is indeed the composition of a rotation (i.e. \textit{the composition first 68 reflections}) and the last reflection.
\end{example}

We have
\begin{theorem} \label{thm: sticknumbersB3 }
    In $\tB$ we have $l(\tB) \le 40$ and $l_3(\tB)  = 42$. These minimums are achieved by trefoils. Hence, $l(3_1) = 40$ and $l_3(3_1) = 14$.
\end{theorem}
\begin{proof}
This result was proven by an exhaustive computational search. In summary, we constructed all possibilities, ruling out what we could using computer programs and some eliminating criteria. After that, we ruled out the remaining cases by looking at them with an interactive program we coded. It allowed us to see, manipulate and explore the different $\tX$ (and also the cubic galleries of $C$). It is avaialable at \cite{CoxeterShapes}.

The tables for the galleries that have order $3$ were created by a tree search in breadth (analogous to what in \cite{mannminimalknottingnumbers} they all a \say{worm} algorithm). We picked the first letter and then added a different letter after it. We eliminated the word if the knot self intersected or if after repeating it three times it didn't close (i.e. not order $3$).

Eliminating criteria we used were of two sorts. Firstly, we need the word to have an even number of letters $D$, so we ignored everything that didn't because it would not close. The reason for this is as follows: $D$ is the letter that corresponds to a change in cube. Since the gallery must return to the first cube it started in, the number of $D$'s must be even. This much is true in any closing gallery. However, in one that is created by repeating a word three times, this word itself must have an even number for otherwise the concatenation will not close. We further organized the lists by increasing number of $D$'s appearing, which is equivalent to bounding the number of cubes the gallery visits. 

Finally, we performed a reduction algorithm that scanned the word for the subwords coming from the defining relations of the Coxeter presentation of $\tB$ and some of the Tits M-moves. This reduced the length of the lists significantly. 

We found the minimal length for nontrivial knots with threefold symmetry are of length $42$ and were produced only by trefoils. An example of which is given in \ref{fig: minimal trefoil with threefold symmetry} above.

To bound the knotting number by $40$ we found the following trefoil.
\begin{center}
    \includegraphics[scale = 0.5]{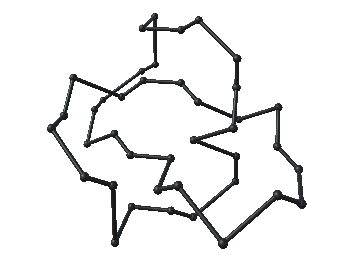}
    \captionof{figure}{$DABCACDACDBCABCACDBACDCBDCBCABDCACABDCBC$}
\end{center} 
These appeared while we searched the lists of knots with up to $8$ letter $D$'s. It visits four cubes.
    \begin{center}
    \includegraphics[scale = 0.4]{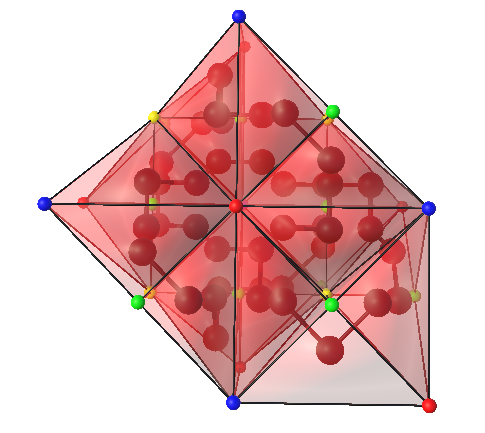}
    \captionof{figure}{The trefoil of length 40 in the four cubes.}
\end{center} 
After this the lists were too long for the machines to produce and for us to check even with the reduction algorithm.
\end{proof}
\begin{remark}
 In \cite{Diao} it is proven that $l(C) = 24$. However, the minimal trefoil presented there is not rotationally symmetric. In \cite{huhohlatticestick}, a rotationally symmetric minima is given, thus proving $l_3(C) = 24$ and $l_{3, C}(3_1) = 8$.
\end{remark}
\begin{example}
   Another trefoil of with $8$ D's that we found through our searches that presents rotational symmetry of order $2$ is the following
    \begin{center}
    \includegraphics[scale = 0.4]{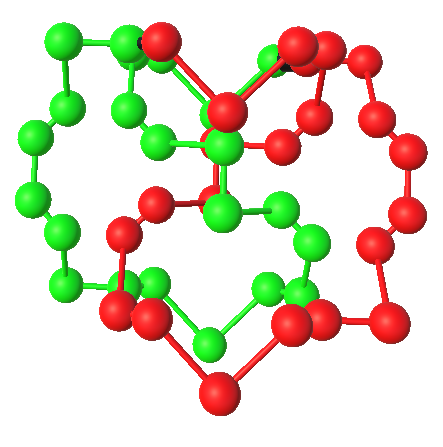}
    \captionof{figure}{$CBACACBDACBCABCDACDCBD$}
\end{center} 
\end{example}

\begin{example}
    We show some trefoils produced from order $3$ words that are interesting.
    \begin{center}
    \includegraphics[scale = 0.5]{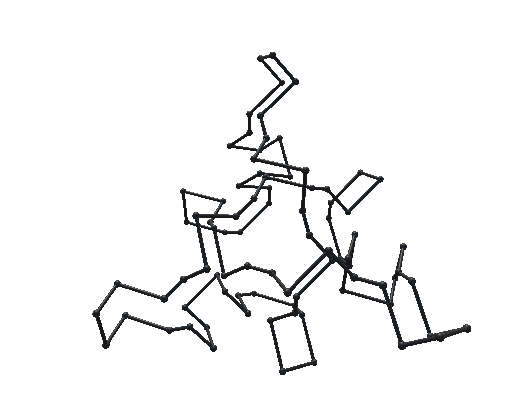}
    \captionof{figure}{$DCADBDACDADBADACACBABDCADCBA$}
    \end{center}
The trefoil above has length $48$ and is created from a word of order $16$. 
A trefoil that is visually deceptive is the following one.
    \begin{center}
    \includegraphics[scale = 0.5]{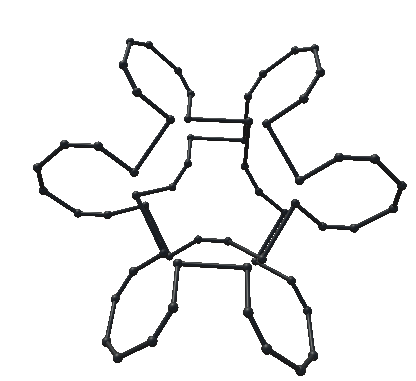}
    \captionof{figure}{$DCDCDCDBDCDCDCBADCBA$}
    \end{center}
At first it looks this trefoil has sixth rotational symmetry. That cannot happen because no such rotation exists in $\tB$. Indeed, the first ten letters are $DCDCDCDBDC$ and the last ten are $DCDCBADCBA$. These are \textit{different} elements of the Coxeter group $\tB$. You can notice the difference in how consecutive \say{octagons} are joined.

\end{example}


\section{Translation from the Cubes to the Pyramids}\label{sec: translation from the cubes to the pyramids}

The cubic lattice  $C$ also admits a threefold rotational symmetry: \textit{it is a rotation of $120^\circ$ along an axis joining two opposite vertices of a cube}. For the purposes of constructing bigger knots with threefold rotational symmetry it is useful to find them first in the cubic lattice and then change them to $\tB$. 

\subsection{Codifying the cubic galleries in the static reference frame}\label{subsec: codifycubes}

Just as we codified the galleries in $\tB$, we want to codify those in the cube. This tesselation is not coming from a Coxeter Group.  However, we will describe a similar codification by how you move across consecutive cubes.

We consider a fixed orientation for $\mathbb{R}^3$ and based in the direction of the positive axis we will decide what is front, right and up. Front is the direction of the positive $x$-axis, right is the direction of the positive $y$-axis and up is the direction of the positive $z$- axis. We will denote by $F, R, U$ these front, right and up directions. Their negatives, which we call back, left and down, by $B, L, D$.

If three unit vectors along the axis are regarded as the three sides in a corner of a unit cube, then the line through the origin and the opposite vertex is an axis for a positive threefold rotation and a negative threefold rotation. The sides of this cube from the origin can be labeled with three of the directions $F, B, R, L, U, D$. We will always suppose this is the case. We emphasize that it is not always the same three vectors the ones that are being identified as the sides of the cube. We shall call this cube the \textbf{reference cube}.

The following is an example of how the description of a gallery changes when we apply such a threefold rotation. Other analogous propositions are obtained by changing the three vectors, or equivalently, changing the octant where the reference cube is.
\begin{proposition}
  Let $g = (s_1, s_2,..., s_k)$ be a cubic gallery starting at some point $p$ and with $s_i\in\{F, B, R, L, U, D\}$. Let $\Phi$ be the positive rotation by $120$ degrees (i.e. the positive threefold rotation) and suppose the reference cube is in the positive octant. Then the gallery $\Phi(g)$ starts at $\Phi(p)$ and is codified by
    \begin{equation*}
        (\sigma(s_1),..., \sigma(s_k)),
    \end{equation*} 
    where $\sigma$ is the permutation $(FRU)(BLD)$. Consequently, the negative threefold rotation is given by $\sigma^2 = (FUR)(BDL)$.
\end{proposition}
\begin{proof}
    The rotation just permutes the sides of the cube in a cyclic fashion of order three. 
\end{proof}

\begin{example}\label{ex: orderthree trefoil in cubes}
The following is a trefoil in a cubic path that presents rotational symmetry. The reference cube has the vectors $R, F, D$ as sides. 
    \begin{center}
    \includegraphics[scale = 0.5]{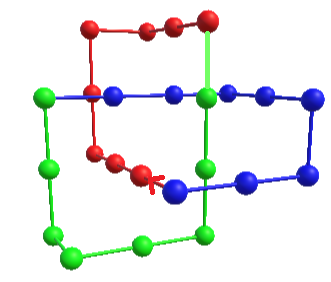}
    \captionof{figure}{A trefoil with threefold symmetry formed by three congruent cubic pieces}
    \label{fig: Cubic gallery of $3_1$ with threefold symmetry.}
    \end{center}
A rotation of $-120$ takes the red piece onto the green piece, which in turns rotates onto the blue one. Lastly, the blue one rotates onto the red one. Thus the rotation is given by the permutation
\begin{equation*}
    (FDR)(BUL).
\end{equation*}
The word describing the red piece is $FFFUURBB$. Applying the permutation we described above we get that the way the pieces rotate is
\begin{equation*}
    {\color{red}FFFUURBB} \longrightarrow {\color{green}DDDLLFUU} \longrightarrow {\color{blue}RRRBBDLL}
\end{equation*}

\end{example}

The following result is proven in the same way as proposition \ref{prop: knotrotational} so we just comment on the essentials.

\begin{proposition}\label{prop: cubesrotation}
    A knot $K$ has threefold rotational symmetry if and only if there exists a cubic gallery
    \begin{equation*}
        g = (c_1, c_2,..., c_k)
    \end{equation*}
    and a threefold rotational symmetry $R$, codified by the permutation $sigma$, of  $C$ such that the polygonal curve created by joining consecutive centers
    of
    \begin{equation*}
        (g, \sigma(g), \sigma^2(g))
    \end{equation*}
    is a knot of the same type as $K$. Notice that $\sigma$ is a product of two disjoint three cycles.
\end{proposition}
\begin{proof}
    That such a decomposition onto three cubic paths which are obtained from one another by a threefold rotation on certain axis, which is also a symmetry of $C$, is done in the same way as in proposition $\ref{prop: knotrotational}$. Once this is achieved, one of the cubes with one of its vertices on its axis can be declared as the reference cube. Once this choice is taken, according to whether the rotation is positive or negative the appropriate permutation $\sigma$ is specified and thus the gallery can be described as
    \begin{equation*}
        (g, \sigma(g), \sigma^2(g)).
    \end{equation*}
\end{proof}
\begin{remark}
    We emphasize that for the static reference frame, when we are codifying galleries which produce knots with threefold symmetry, the first letter represents the direction in which the gallery moves from the \say{last piece} into the \say{first piece}. In particular, the length of the codifying word coincides with the length of the piece that is being rotated.
\end{remark}
 
\subsection{Codifying the cubic galleries in dynamic reference frame}\label{subsec: codifycubestnb}

The static reference frame is useful to describe individual pieces in the cubic lattice. However, our goal is to describe paths in $\tB$. What we would want is thus a translation from this cubic word into one in the generators of $\tB$. The problem that we face is that the three words that compose the cubic path are not really the same, they differ by the permutation $\sigma$ applied to its individual letters. 

In order to describe the cubic path in a way that it makes sense to concatenate three times the same word and get the knot we are looking for we must use a reference frame adjusted to the cube we are as opposed to using the one given by the euclidean space the live is in. This is the classical choice between a static reference frame versus a dynamic reference frame, which we call the TNB frame (in differential geometry,  it is called the Frenet-Serret Frame.) 

We now describe how we assign our TNB frame given a closed cubic path. We will suppose the path is oriented. The only thing that we have to do is to define the vector frame at an initial point and the tangent vector $T$ at each point of the curve. For the rest of the path, orthonormality and fixed orientation of the basis, determines the other vectors $N$ and $B$. We define the vectors at the initial point following the next convention.

\begin{convention}\label{conv: TNB description}
The initial $TNB$ frame is defined with respect to the first point of one of the pieces as follows:    
    \begin{description}
    \item[Tangent vector $T$:] The unitary vector that goes in the direction the path is oriented. 
    \item[Normal vector $N$:] The unitary vector, orhtogonal to $T$, that points in the direction of the first turn in the path. Such a turn must exist because the path closes.
    \item[Binormal vector $B$:] The unitary vector $B = T\times N$.
\end{description}
\end{convention}

Our convention to codify a path will be similar to the static reference frame, except that the notion of up, down, front, back, right and left depends on the TNB frame. 

\begin{convention}\label{conv: FUR in TNB}
 In a $TNB$-frame we declare \textbf{front} to be the direction of $T$, \textbf{up} to be the direction of $N$ and \textbf{right} to be the direction of $B$. 
 To emphasize that we are using the $TNB$ reference frame we use lower case letters to denote a path of cube in this direction.
\end{convention}

Because the curve follows a cubic path, each of the three vectors $T, N, B$ point to a face of the cube. Thus the $T, N, B$ vectors track how the cube moves as it follows the curve. This introduces a subtlety that we will explain by revisiting the example of the trefoil.

\begin{example}\label{ex: orderthree trefoil in cubes pt1}
We have once more the trefoil in the cubes with threefold symmetry. The curve is oriented in the direction of the red arrow as shown in figure \ref{fig: Cubic gallery of $3_1$ with threefold  symmetry.} above. We suppose we shall start our description of the path in the first red vertex. 
The $TNB$ frame shown is the starting one. Thus the description of the red path is
\begin{equation*}
    ffufruf.
\end{equation*}
Notice however that the description of the green path and the blue path is exactly the same because the choices of the starting TNB frame described above do not change under a threefold rotation. Thus at first we might be tempted to describe this trefoil, in the TNB frame as concatenating
\begin{equation*}
    {\color{red}ffufruf}{\color{green}ffufruf}{\color{blue}ffufruf}
\end{equation*}
This clearly does not represent the above figure because at the merging of the pieces, it instructs to keep the direction (i.e. it has $f$) three times, but that is not what the path does. When it starts following the green part, the path turns. 
\end{example}

The discrepancy we have just witnessed is introduced by the torsion, i.e. \textit{how much the cube twists as it follows the path}. At the end of the path we have a corresponding $TNB$ frame which has been produced by following the path. However, that last point of the path is where we will start following the rotated path and we can put an initial $TNB$ frame following the rules above. There is no reason why these two frames coincide, and as we see in the above example, they might not. 

We thus define a linear transformation  $\tau:\mathbb{R}^3\longrightarrow \mathbb{R}^3$ as the one that sends the final $TNB$ frame, by following the first path, to the initial TNB path, constructed by the rules above with respect to the rotated path.  We emphasize that this $\tau$ is independent on which of the three rotated paths we follow, because what it measures the torsion of the path, which is invariant under rotations.

\begin{proposition}
   The linear transformation $\tau$ is a positive rotation (i.e. $\tau\in\mbox{SO}(3, \mathbb{R})$). 
\end{proposition}
\begin{proof}
    The linear transformation $\tau$ sends a positve orthonormal basis, the final TNB basis following the first path, to another positive basis, the TNB basis constructed from the rules above. Thus it is an element of $SO(3, \mathbb{R})$. 
\end{proof}

Notice that it cannot be that only one vector is misplaced. Indeed, if two of them are correct (i.e. fixed by $\tau$) then the third one is also fixed because it is the cross product of the other two. Thus, either only one vector is fixed or none are. 

\begin{definition}
We say that $\tau$ is a \textbf{bar gluing} if $\tau(T) = T$ and a \textbf{corner gluing} if $\tau(T)\neq T$.
\end{definition}
\begin{remark}
These names are suggestive of how the three pieces glue. In a bar gluing, because $T$ is correct, the end of one piece merges with the start of the next to produce one straight bar of the cubic gallery of the knot. On the other case, the end of one piece glues with the next one creating a turn of the cubic gallery. 
\end{remark}

\begin{example}\label{ex: orderthree trefoil in cubes pt2}
Continuing the previous example, we see by inspection that comparing the two $TNB$ in the last red vertex, we see that $N$ is already pointing in the right direction. $T$ and $B$ must be rotated by $90$ degrees around the axis of $N$ in the positive direction. That is:
\begin{equation*}
    \tau(T) = -B, \tau(N) = N, \tau(B)= T.
\end{equation*}
This is a corner gluing. Indeed, looking at the figure \ref{fig: Cubic gallery of $3_1$ with threefold symmetry.}, we see that different pieces create corners.
\end{example}

\subsection{Translating from $C$ to the $\tB$}\label{subsec: translating}

Finally, we can produce words in $\tB$. We will suppose given a knot with threefold symmetry of which we know a cubic word, given by a gallery in TNB frame, together with the corresponding $\tau:\mathbb{R}^3\longrightarrow \mathbb{R}^3$.

What we do is to replace each $f, b, u, d, r, l$ by an appropriate gallery of pyramids, that is a word in $A, B, C, D$, that mimics the movement given by the $TNB$ frame. Then, if needed, we also add a words, also in the letters $A,B,C,D$, to perform the correction given by $\tau$. 

We bring attention to the following point, which will be very relevant to our later codification:
\begin{description}
    \item[Cube parity:] When we are dealing with cubic galleries, all cubes are the same. However, when we work with $\tB$, the cubes of which this tesselation comes from have a parity attached to them. That is, not all of them are translations of each other. The difference resides in the alternating pattern of the vertices (of the cubes) labeled $A$ and $B$. Each cube is adjacent to six other cubes of the opposite parity. \textit{We shall suppose we have decided on some fixed parity for the cubes as explained above.}
\end{description}

Let $K$ be a knot with threefold symmetry in the cubic lattice and suppose $\tau$ is a bar-gluing. Then, because $T$ is correct, when we conclude following the first piece, the cube that contains the first cube of the following piece is in front of it.  However, the direction of the piece is indicated wrong, because the first turn goes in the $N$ direction which is, by assumption, wrong. However,this can be fixed by rotating the $TNB$ frame, on the $T$ axis by a positive rotation by $90$, a rotation by $180$ or a negative rotation by $90$.

Thus, we can codify the way we glue consecutive pieces in the case of a bar gluing by correcting the problem introduced by the torsion and then applying $f$.
   
We are ready to give the translation between $TNB$ and $\tB$. We begin with a definition:
\begin{definition}
    Let $K$ be a knot with threefold symmetry given by a path of cubes of length $3k$ with threefold rotational symmetry. Take a subpath of length $k$ and use it to construct the gallery $g = (c_1,..., c_{k})$ and rotation $\tau:\mathbb{R}^3\longrightarrow\mathbb{R}^3$. Suppose $\tau$ is a bar gluing. 
    We define the \textbf{type of $\tau$}, denoted by $s$ or $s(\tau)$, by the following table
    \begin{center}
    \begin{tabular}{|c|c|}
    \hline
       \textbf{Bar Gluing $\tau$} & $s$  \\
    \hline
       No rotation needed  & $1$\\
    \hline
        Positive 90 degrees & $i$\\
    \hline
         180 degrees & $-1$\\
    \hline
          Negative 90 degrees& $-i$\\
    \hline
    \end{tabular}
\end{center}
That is, identifying the $NB$ plane with the complex plane, $s(\tau)$ is the complex number corresponding to the restriction to the $NB$-plane of the inverse of $\tau$.
\end{definition}

Finally, the translation is given as follows
\begin{theorem}\label{theorem: piece change from cubes to pyramid}
Let $K$ be a knot with threefold symmetry and let $g = (c_1, c_2,...., c_k)$, where $c_i\in\{f, u, d, r, l\}$, be a $TNB$-codification of the cubic piece to be rotated (notice $b$ does not appear because a path doesn't go back on itself).
 Let $\tau:\mathbb{R}^3\rightarrow\mathbb{R}^3$ be the corresponding rotation and suppose it is a bar gluing of type $s$. Construct a gallery in $\tB$ as
\begin{equation*}
    T(g):= (T(c_1), T(c_2),..., T(c_k), T(s)),
\end{equation*}
where $T$ is defined in $f, u, d, r, l$ as given by the following table
\begin{center}
    \begin{tabular}{|c|c|c|}
    \hline
       \mbox{Letter} & \mbox{Even cube} & \mbox{Odd cube} \\
    \hline
        $f$& \multicolumn{2}{|c|}{$ABCDC$}  \\
    \hline
        $u$& \multicolumn{2}{|c|}{$ABCABCDC$}  \\
    \hline
        $d$& \multicolumn{2}{|c|}{$DC$}  \\
    \hline
        $l$& $ACDC$& $BCDC$ \\
    \hline
        $r$& $BCDC$& $ACDC$ \\
    \hline
    \end{tabular}
\end{center}
and $T(s)$ is defined by the following table
\begin{center}
    \begin{tabular}{|c|c|}
    \hline
       \textbf{s} & $T(s)$  \\
    \hline
       $1$  & $ABCDC$\\
    \hline
         $i$ & $ACBCDC$\\
    \hline
          $-1$ & $BACABDC$\\
    \hline
         $-i$ & $BCACDC$\\
    \hline
    \end{tabular}
\end{center}
Then the gallery in $\tB$ given by
\begin{equation*}
    (T(g), T(g), T(g))
\end{equation*}
knots as $K$. In particular, $T(g)$ is a word of order $3$ that represents $K$.
\end{theorem}
\begin{proof}
   Let us pick a cube where we will start reading the $TNB$ frame. We emphasize that we are not given the cubic gallery itself but rather a codification in $TNB$ frame with bar gluing $\tau$. Thus, we are free to choose the directions which will be $T$ and $N$ in this first cube.
   
   Without loss of generality we will suppose this cube is even (recall this means a fixed choice of which vertices of the cube are $A$ and which are $B$). Flattening this cube we pick a pyramid whose $ABC$-face lies on the cube face where $-N$ is pointing and whose $AB$ edge lies on the face where $-T$ is pointing (i.e. the back face). We call this pyramid the \textit{pivot pyramid}.
   
   The five words given by $f, u, d, l, r$ end with $DC$. Without these last two letters, the pyramid moves withing the same cube. If we flatten the cube the final \say{destinations} of $\mathcal{C}$ along these paths are as shown in the following picture. Notice that the location of $\mathcal{C}$ coincides with $d$.
    \begin{center}
    \includegraphics[scale = 0.5]{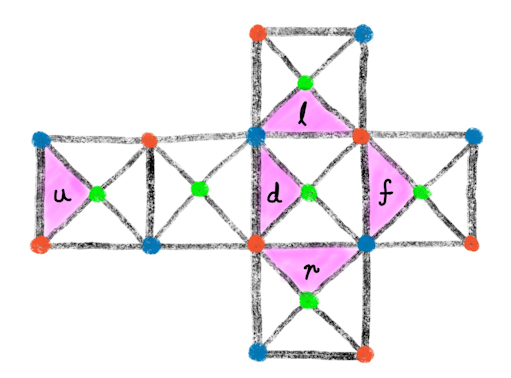}
    \captionof{figure}{Locations of the pivot pyramid before it changes cube with $DC$.}
    \end{center}
   Notice that the cube where each final pyramid, after adding the $DC$ in each word, lies in a different cube and satisfies the following two properties:
   \begin{enumerate}
       \item The cube where each letter $m\in\{f, u, d, r, l\}$ indicates in the $TNB$ movement coincides with the cube to which the pivot pyramid moved within the corresponding gallery.
       \item Taking into account the change in $TNB$ cube after the first move, the pivot pyramid $\mathcal{C}$ moved to the pyramid which is the pivot pyramid of the next cube.
   \end{enumerate}
   We show this in the next image except for $u, f, r, l$. (For $d$ the image is less clear because the face lies in a face between the cubes but the argument is the same).
    \begin{center}
    \includegraphics[scale = 1]{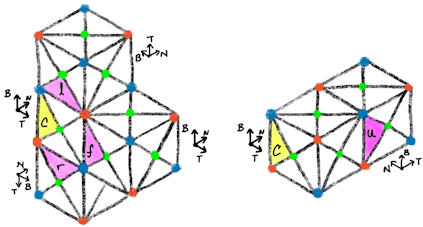}
    \captionof{figure}{The pivot pyramid changes in the same was as the $TNB$ frame.}
    \end{center}
   The above argument is for a change from an even to an odd cube. The other possibility is analogous. Thus, by performing the steps one after the other we follow with a path of pyramids the cubes described by the codification given by the dynamic frame. 
   
   Now we handle the cubes where the pieces glue. By assumption, $\tau$ is a bar gluing. This means that in this last cube, the pyramid has its $AB$ edge and an edge of the back face. However, it might be that its $ABC$ face does not point in the $-N$ direction. Thus, we move \textit{within} this cube the pivot pyramid. There are four possible pivot pyramids for this correction, including the case where no fix is required.
    \begin{center}
    \includegraphics[scale = 1]{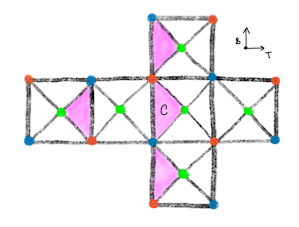}
    \captionof{figure}{Possible locations for the new pivot pyramid.}
    \end{center}
   
   Once repositioned, because it is a bar gluing, and we have fixed $\tau$ we perform $f$. Composing the words given by fixing $\tau$ with the word for $f$ we get:
   \begin{align*}
        (ACA)(ABCDC) &= ACBCDC,\\
       (BCBACA)(ABCDC) &= BACABDC,\\ 
        (BCB)(ABCDC) &= BCACDC.
   \end{align*}
   Notice that the last cube of any piece of a threefold rotationally symmetric knot must be odd. This is why $BCB$ is the positive rotation and $ACA$ the negative rotation. 
   
   Performing this operation we move to the first cube of the new piece matching the $TNB$ it must have. In particular, the pivot pyramid has moved to coincide with the appropriate pivot pyramid of the next rotated piece. Now we can repeat the same process twice and at the end it will close because the pivot pyramid in the first cube is determined by the starting $TNB$ frame.

    By construction, the gallery $(T(g), T(g), T(g))$ closes and thus this is a word of order three. Furthermore, the knot it represents is $K$ because one can deform the path created by joining centers of pyramids to that of the centers of cubes one cube at a time. This is similar to the argument in Proposition \ref{prop: knotrotational}. We leave the details to the reader.
    
    This concludes the proof.
\end{proof}

For all practical purposes, we do not lose generality by restricting our constructions to bar-gluing. Indeed, suppose in $\tB$ we have a word $w$ of order three represented by a gallery
\begin{equation*}
    (s_1,..., s_m, t_1,..., t_n).
\end{equation*}
Let $w = s_1...s_mt_1...t_n$. We can \textit{shift} the gallery that will be rotated to construct the knot. Concretely, this will change the shape of the piece, but the final gallery that constitutes the complete knot doesn't change. This can always be achieved because if
\begin{equation}
    s = s_1...s_m, t = t_1...t_n,
\end{equation}
then from $w^3 = 1$, we obtain
\begin{equation*}
    ststst = 1,
\end{equation*}
and thus conjugating by $s$ one gets
\begin{equation*}
    s^{-1}(ststst)s = 1
\end{equation*}
that is, 
\begin{equation*}
    tststs = 1.
\end{equation*}
Thus, the new piece is given by the gallery $(t_1,..., t_n, s_1,..., s_m)$, and the new starting chamber is $s_1...s_m(\mathcal{C})$. 

Similar arguments can be made for the cubic lattice. In particular, if one is interested in working with a bar-gluing one can shift the piece to start in the middle of a straight bar as opposed to a corner. In the process of construction that we shall follow this can also be guaranteed. Furthermore, one can also artificially enlarge the length of the bars by reducing the scale of the cubes (as done in Proposition \ref{prop: knotrotational}, for example).

In conclusion, we shall shift the cubic gallery to get a bar gluing, translate into $\tB$ and then unshift in the Coxeter word to follow the original cubic piece.

\subsection{Comparing lattices}\label{subsec: comparisonlattices}

We will end this section by briefly discussing certain bounds on the length in $\tB$ and in the cubic lattice $C$.  

\begin{proposition}
    Let $g$ be a gallery in $\tB$ that comes from a cubic gallery $g_C$ in $C$, as explained in subsection \ref{subsec: translating}. Then
    \begin{equation*}
        l_{\tB}(g) \ge 2l_C(g_C)
    \end{equation*}
\end{proposition}
\begin{proof}
 For a gallery 
\begin{equation*}
    g = (s_1,. s_2,..., s_k)
\end{equation*}
let $N_D(g)$ be the number of indices $1\le i\le k$ such that $s_i = D$. We have discussed before that $D$ represents the change between cubic faces. Thus if $g$ never visits any cube more than twice (i.e. it goes in and out of each cube exactly once), then the cubic gallery created by joining consecutive centers of cubes visited by $g$, which is a gallery in $C$, has length $N_D(g)$. Applying this to $g_C$, we conclude that
\begin{equation*}
    N_D(g) = l_C(g_C).
\end{equation*}
Because we must not have two $D$'s in a row, there must be at least one different letter in between two consecutive $D$. Thus $g$ actually has $N_D(g) - 1$ separating the $D's$. We conclude,
\begin{equation*}
    l_{\tB}(g)\ge 2l_C(g_C) - 1.
\end{equation*}
However, this gallery closes because it represents a knot, its length must be even. Thus,
\begin{equation*}
    l_{\tB}(g)\ge 2l_C(g_C),
\end{equation*}
as desired.
\end{proof}

\section{Order Three Galleries}\label{sec: Order three galleries}

\subsection{Example of the construction for $9_{47}$}

As the exhaustive searches explained in section \ref{sec: knotting length in the Coxeter Lattices} show, the amount of galleries that must be studied to find specific knots in $\tX$ is very large. However, if we look for knots that admit certain symmetries then the lists are smaller. Furthermore, these symmetries can be used to \textit{construct} the words we are looking for as opposed to find for them. Usually, a symmetry that a knot may have is not visible from its knot diagrams. Even when the knot diagram shows this symmetry, it is unclear how to produce an actual physical knot with that symmetry from it.

However, we can nevertheless use these to produce larger knots with threefold symmetries. We will explain the process with $9_{47}$. A plane representation of this knot that shows the threefold symmetry is the following:
\begin{center}
    \includegraphics[scale = 0.3]{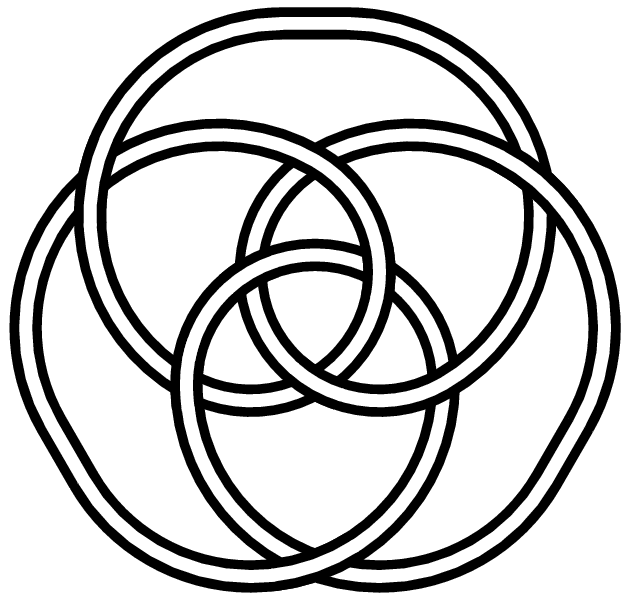}
    \captionof{figure}{$9_{47}$ diagram with rings and threefold symmetry.}
\end{center}
The crossings are arranged in rings around the center of the threefold symmetry. The figure in the central circle as well as those in each annulus is preserved by the rotation.
\begin{center}
    \includegraphics[scale = 0.3]{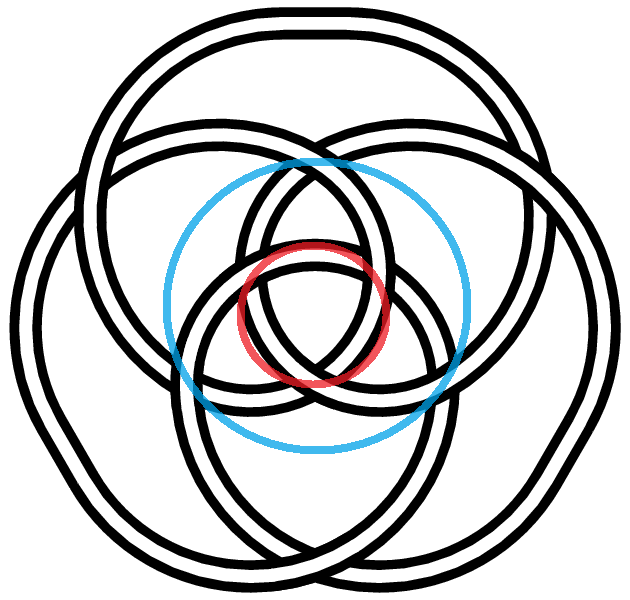}
    \captionof{figure}{$9_{47}$ diagram with threefold symmetry.}
\end{center}
There are three pieces: the inner circle, the annulus and the outside. Using those, we see that it can be constructed inductively as follows.
\begin{center}
    \includegraphics[scale = 0.5]{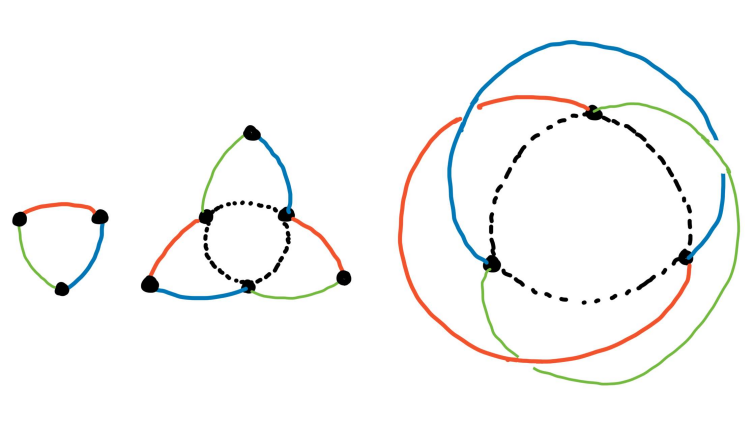}
    \captionof{figure}{Different parts of the $9_{47}$ according to ring.}
\end{center}
Our strategy will succeed if we can do the following two things:
\begin{enumerate}
    \item Construct each figure with cubes in such a way that the rotational symmetry is preserved in the same axis. 
    \item Merge different stages together by preserving the top and under crossings. 
\end{enumerate}
This works as long as the smallest piece can be constructed as the base case. In this situation, this can be done and it consists of three strings.
\begin{center}
    \includegraphics[scale = 0.5]{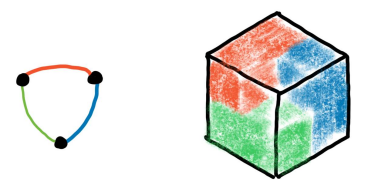}
    \captionof{figure}{First stage of $9_{47}$ and its cube.}
\end{center}
This can be done with a $2 \times 2 \times 2$ cube and it preserves the threefold symmetry. Furthermore, the upper and lower crossings at the first circle are preserved. 

For the next stage, we forget what is inside this box. We just keep the direction at which different strings must go out of the box. Then we create the required tangle (in this case it's the trivial blue-red one) and rotate 
\begin{center}
    \includegraphics[scale = 0.5]{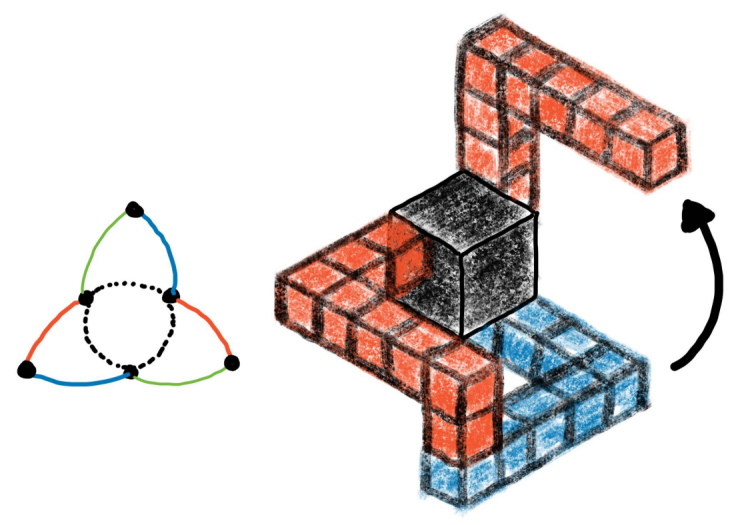}
    \captionof{figure}{Second stage of $9_{47}$ over the first stage cube.}
\end{center}
Because the new strings go into different disjoint regions of the cube the rotational symmetry is possible without problem. Notice how red is on top of blue, as it is required from the diagram. Moreover, because the central cube must be at the center of the next one, so that they share the axis of symmetry, we see that the new cube must extend three cubes in each direction. So it is an $8\times 8\times 8$ cube.

For the last stage, we forget what is inside this black box, except where the pieces go outside. This time the strings do interact with each other so we must be mindful of not crashing when rotating and of respecting crossings.
\begin{center}
    \includegraphics[scale = 0.5]{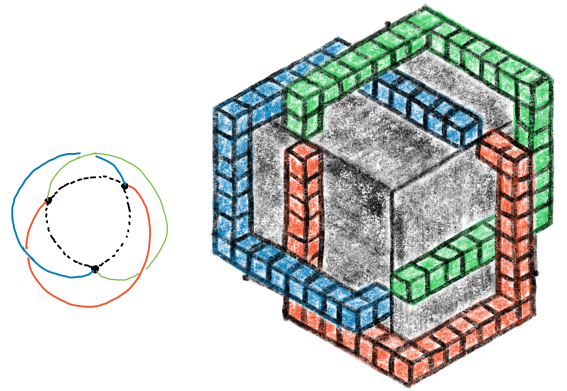}
    \captionof{figure}{Third stage of $9_{47}$ over the second stage cube.}
\end{center}
Notice that we rotated the last image so that it's easier to see that  the crossings and colors are respected. The farthest we moved away from each side was $2$, thus the final figure lies inside a $12\times 12 \times 12$ cube.

By connecting together the different pieces of the same color we reach the piece that when rotated produces the knot we want. For example, connecting the red parts gives the following figure.
\begin{center}
    \includegraphics[scale = 0.5]{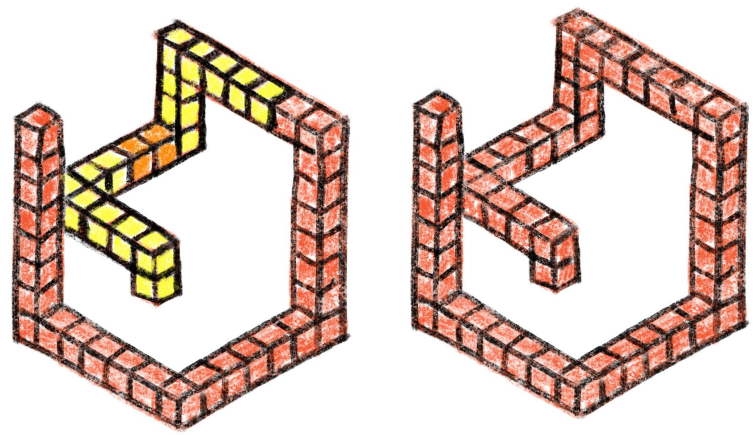}
    \captionof{figure}{Piece that produces $9_{47}$ colored by stage (left) and without distinction (right).}
    \label{fig: piece of 9_47}
\end{center}
Now, using the threefold rotation of the $12\times 12 \times 12$ we have been using to make sure these pieces indeed rotate appropriately we can build the knot with the three pieces. 
\begin{center}
    \includegraphics[scale = 0.5]{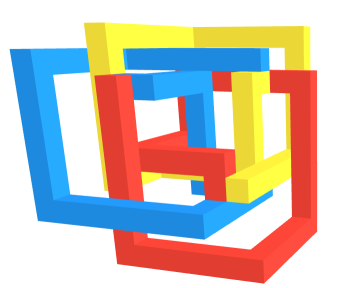}
    \captionof{figure}{The three pieces that create $9_{47}$}
    \label{fig: 9_47 constructed}
\end{center}
We now shift the pieces to assure $\tau$ is a bar gluing. We do it as shown in the following picture
\begin{center}
    \includegraphics[scale = 0.2]{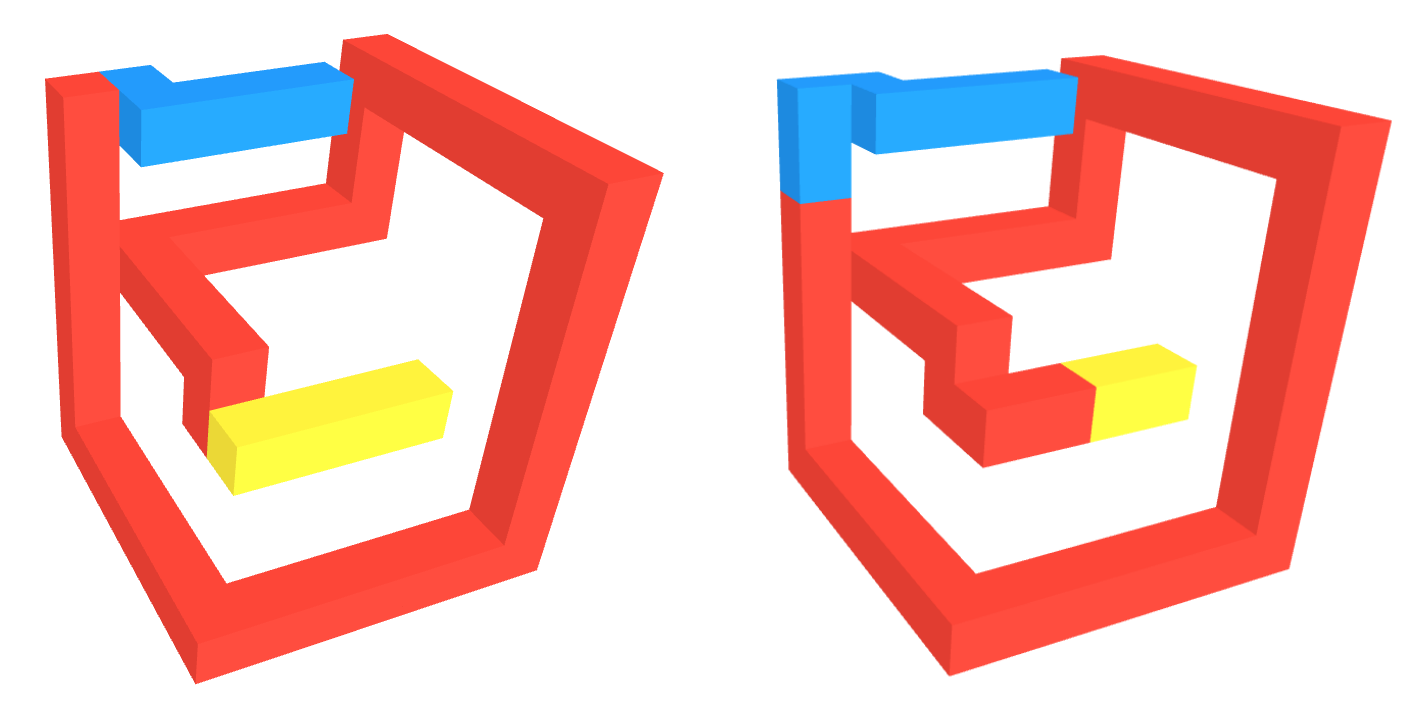}
    \captionof{figure}{The new piece to be rotated}
    \label{fig: 9_47 new piece}
\end{center}
Following the direction of our new shifted piece from yellow to red, we see that the $TNB$ frame path is
\begin{equation*}
    fulrfffufffflffufffffuffffffflfffffuffffflffffff
\end{equation*}
Comparing with the rotated piece that follows (i.e. the blue one) the rotation $\tau$ is the identity. This leads to the following word of order three, using Theorem \ref{theorem: piece change from cubes to pyramid} in the $\tB$ Coxeter Complex. 
\begin{align*}
\;ABCDCABCABCDCBCDCACDCABCDCABCDCABCDCABCABCDCABCDCA\\
\;BCDCABCDCABCDCACDCABCDCABCDCABCABCDCABCDCABCDCABCD\\
\;CABCDCABCDCABCABCDCABCDCABCDCABCDCABCDCABCDCABCDCA\\
\;BCDCACDCABCDCABCDCABCDCABCDCABCDCABCABCDCABCDCABCD\\
\;CABCDCABCDCABCDCACDCABCDCABCDCABCDCABCDCABCDCABCDC
\end{align*}

We can now undo the shifting so that this $\tB$ piece follows exactly the cubic pieces. Doing this leads to a word of order three given in part (4) of theorem \ref{thm: words for order three for cross 9} below.
\subsection{Prime knots of $9$ crossings with threefold symmetry}

The other prime knots with threefold symmetry of $9$ crossings are $9_{35}$, $9_{40}$ and $9_{41}$. Their respective diagrams with threefold symmetries are given below:
\begin{center}
    \includegraphics[scale = 0.3]{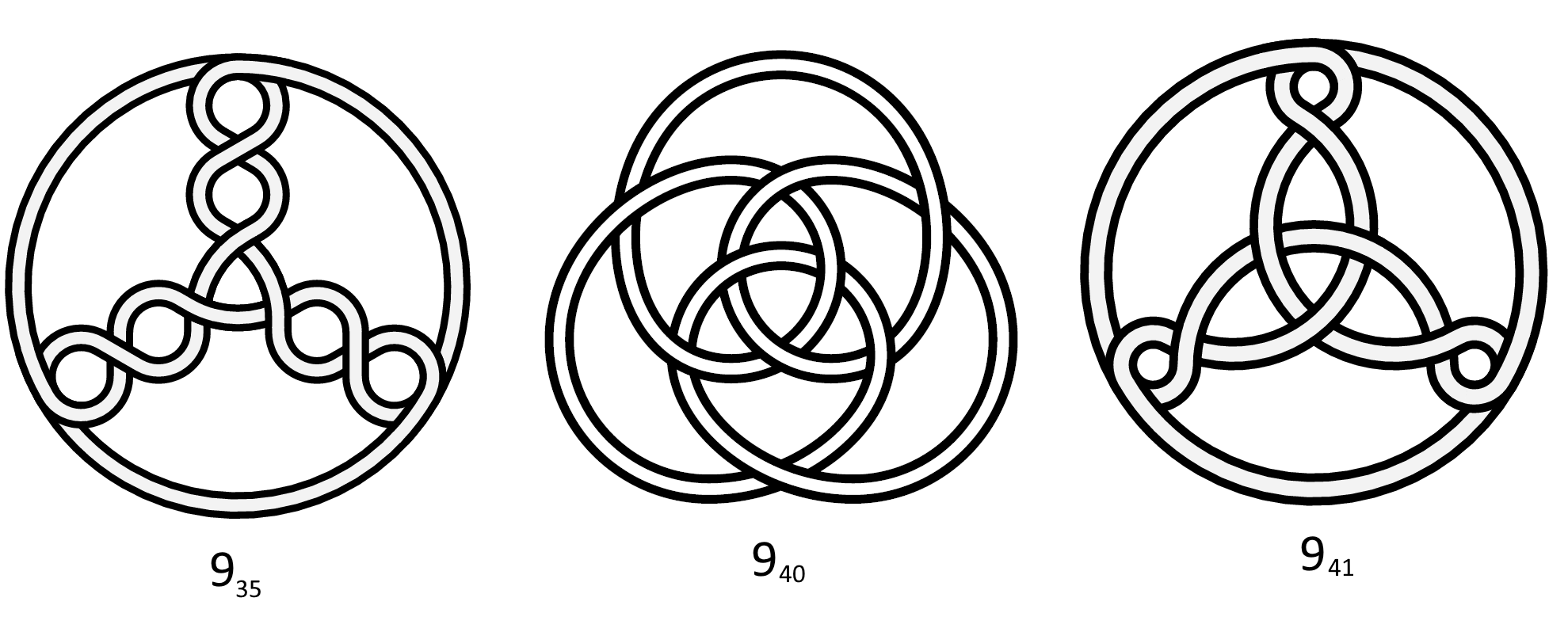}
    \captionof{figure}{Three rotationally symmetric diagrams of $9_{35}, 9_{40}$ and $9_{41}$ taken from Wikipedia.}
\end{center}

Following up the strategy same strategy we did for $9_{47}$ we get the following
\begin{theorem}\label{thm: words for order three for cross 9}
In $\tB$ we have
\begin{enumerate}
    \item The gallery
    \begin{align*}
        \;ABCABCDCABCDCDCBCDCABCDCACDCA\\
        \;BCABCDCBCDCABCDCABCDCABCDCBCD\\
        \;CDCABCABCDCABCDCACDCABCABCDCA\\
        \;BCDCABCABCDCACDCABCDCACDCABCA\\
        \;BCDCABCDCABCDCABCDCABCDCABCDC\\
        \;BCDCABCDCABCDCABCDCABCDCABCDC
    \end{align*}
    has order three and length $174$. It represents $9_{35}$.
    
    \item The gallery
    \begin{align*}
        \;&ABCABCDCBCDC\textbf{BCACDC}ABCDCABCDCABCDCACDC\\
        \;&ABCDCABCDCABCDCABCDCACDCABCDCABCDCABCDC\\
        \;&BCDCABCDCABCDCDCABCDCABCDCBCDCABCDCABCDC
    \end{align*}    
    has order three and length $116$. It represents $9_{40}$. The highlighted part represents the rotation introduced to fix the torsion.
    
    \item The gallery
    \begin{align*}
        \;ABCDCABCDCABCDCDCABCDCABCDCABCD\\
        \;CABCDCABCDCABCDCABCDCABCDCBCDCA\\
        \;BCDCABCDCABCDCACDCDCABCDCABCDCA\\
        \;BCABCDCABCDCABCDCABCDCDCABCDCAB\\
        \;CDCACDCABCDCABCDCABCDCDCABCDCBC\\
        \;DCABCDCABCDCBCDCABCDCABCABCDCDC
    \end{align*}
    has order three and length $186$. It represents $9_{41}$.

    \item The gallery
    \begin{align*}
    \;ACDCBCDCABCDCABCDCABCDCABCABCDCABCDCABCDCABCDCABCD\\
    \;CBCDCABCDCABCDCABCABCDCABCDCABCDCABCDCABCDCABCDCAB\\
    \;CABCDCABCDCABCDCABCDCABCDCABCDCABCDCABCDCBCDCABCDC\\
    \;ABCDCABCDCABCDCABCDCABCABCDCABCDCABCDCABCDCABCDCAB\\
    \;CDCBCDCABCDCABCDCABCDCABCDCABCDCABCDCABCDCABCABCDC
    \end{align*}
    has order three and length $250$. It represents $9_{47}$.
\end{enumerate}
\end{theorem}
\begin{proof}
Just as we did for $9_{47}$, we created a cubic piece that when rotated by $120$ degrees along certain axis and certain direction produces the desired knot. Due to the fact that there is a lot of choices in the process, these pieces are highly nonunique. Thus, we do not show the piece construction. Instead, we give the cubic word in static frame. In all of them the permutation is $(FUR)(BDL)$.
\begin{center}
    \begin{table}[h!]
        \centering
        \begin{tabular}{|c|c|}
        \hline
           Knot  & Static Reference Frame Word Code \\
        \hline
            $9_{35}$ &  $FFRDDLFDDDDFLFFURRDFFDLLLLLLBBBBBB$\\
        \hline
            $9_{40}$ &  $FRUUUUUBBBBBDDDDFFFLLLUU$\\
        \hline
            $9_{41}$ &  $UFFFFDDDDDDDDDRRRRUBBBUUUUBBBRRRRDDBBBUU$\\
        \hline
        \end{tabular}
        \caption{The code for the pieces of $9_{35}, 9_{40}$ and $9_{41}$.}
        \label{tab:my_label}
    \end{table}
\end{center}
We recall that the first letter of the Static Reference Frame Word Code represents the direction in which the last cube of the last piece glues to the first cube of the first piece.
These pieces look as shown in the figure below
\begin{center}
    \includegraphics[scale = 0.6]{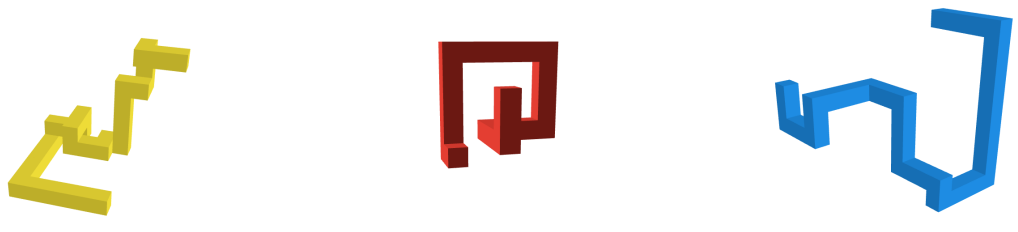}
    \captionof{figure}{The pieces of $9_{35}, 9_{40}$ and $9_{41}$, from left to right, respectively.}
\end{center}
When rotated, to create the cubic knot with threefold symmetry, we obtain the following cubic galleries that knot in the desired way.
\begin{center}
    \includegraphics[scale = 0.5]{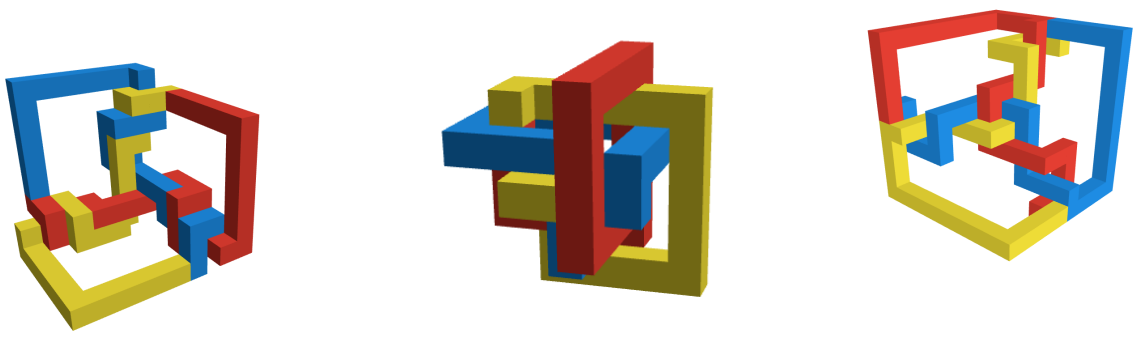}
    \captionof{figure}{The knots $9_{35}, 9_{40}$ and $9_{41}$ with threefold symmetry created from the corresponding pieces.}
\end{center}
We now proceed to the $TNB$ frame. We shifted the piece to guarantee that in all cases $\tau$ is a bar gluing. The next table shows the dynamic reference frame we used and the specific $\tau$ obtained. In all the cases the piece shifting was done just by \say{moving} the piecesome cubes into the next bar.
\begin{center}
    \begin{table}[h!]
        \centering
        \begin{tabular}{|c|c|c|}
        \hline
           Knot  & $TNB$ word for shifted piece & $\tau$\\
        \hline
            $9_{35}$ & $fufdlflurffflduflufurfluffffflffff$ & $\tau = 1$\\
        \hline
            $9_{40}$ & $fffflfffflffflffdfflfful$ & $\tau(N) =B, \tau(B) = -N$.\\
        \hline
            $9_{41}$ &  $udfffdfffffffflfffldffufffdffrfffdflfflf$ & $\tau = 1$\\
        \hline
        \end{tabular}
        \caption{The code for the pieces of $9_{35}, 9_{40}$ and $9_{41}$ in $TNB$ frame.}
        \label{tab:my_label}
    \end{table}
\end{center}
Finally, with Theorem \ref{theorem: piece change from cubes to pyramid}, we converted into the $\tB$ gallery and then we unshifted for the gallery to fit the precise shape of the original figure. This concludes the proof.
\end{proof}

\section{Questions}\label{sec: questions}

\noindent We bring this paper to an end by raising three questions inspired by our work.

\subsection{Question 1: Triviality of knots}

The first one relates to our searches in $\tB$. As we have said above, great part of our work was checking lists of options that were not eliminated by the reduction process or other criteria. In many cases, visually the triviality of the corresponding knot was evident, as the knot was clearly a wide circle

\begin{question}
    Is there a criteria to guarantee triviality of a gallery from the word itself that describes the gallery? 
\end{question}

For example, the word $CABCDCBACBACDACBCDCBCDACBA$ produces the following trivial knot
\begin{center}
    \includegraphics[scale = 0.3]{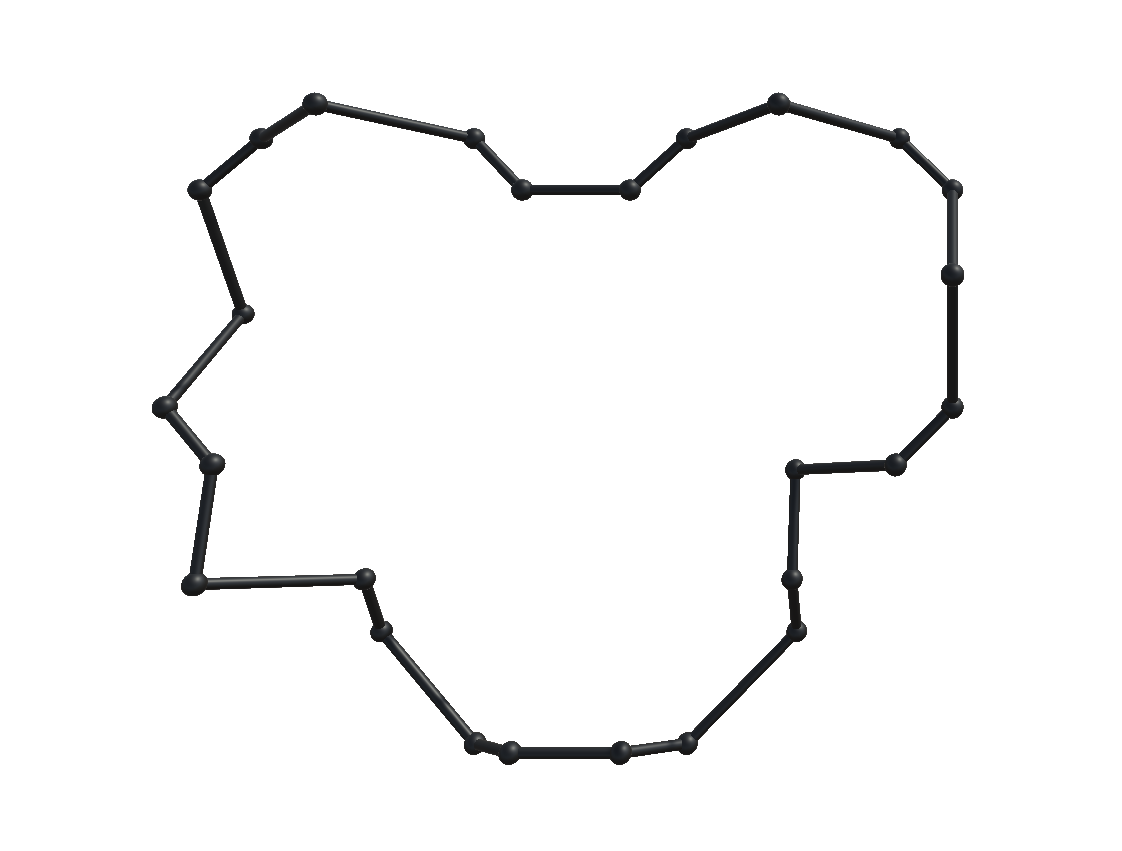}
    \captionof{figure}{$CABCDCBACBACDACBCDCBCDACBA$}
\end{center}
Looking at it, it is obviously a trivial knot. However, from the word itself we cannot distinguish this from any other word. Of course, this is a question that makes sense for any set of conditions with which we construct knots from sticks. However, there is a vast literature on the geometry of Coxeter complexes and their representation theory. It might be possible to find such a criteria in some of these cases.

\subsection{Question 2: Symmetries and symmetric minima}

Our second question is about the symmetries possessed by the lattices where our knots are constructed. As we have discussed, in many situations the knots that present these symmetries are not minimas of the different stick numbers. Thus we ask the following

\begin{question}
    Is there a criteria to determine from the particular lattices, and how their symmetries act on them, that minimal knots cannot posses certain of these symmetries? 
\end{question}
In all the cases we know this happens is because we first find the minimal stick number and then, because it is not divisible by some appropriate number, we rule out the symmetry from manifesting. The point of the above is to predict such minima do not have certain symmetries without knowing the explicit value of the stick numbers.

To give a particular example in a different lattice, we return to the (sh) lattice. This is the honeycomb lattice and it admits three rotational symmetry. The following is a trefoil with threefold symmetry in (sh). In the picture we put two consecutive layers of the honeycomb lattice, without putting the complete edges that join the parallel faces. We do this to make clear how the trefoil is constructed. 

\begin{center}
    \includegraphics[scale = 0.2]{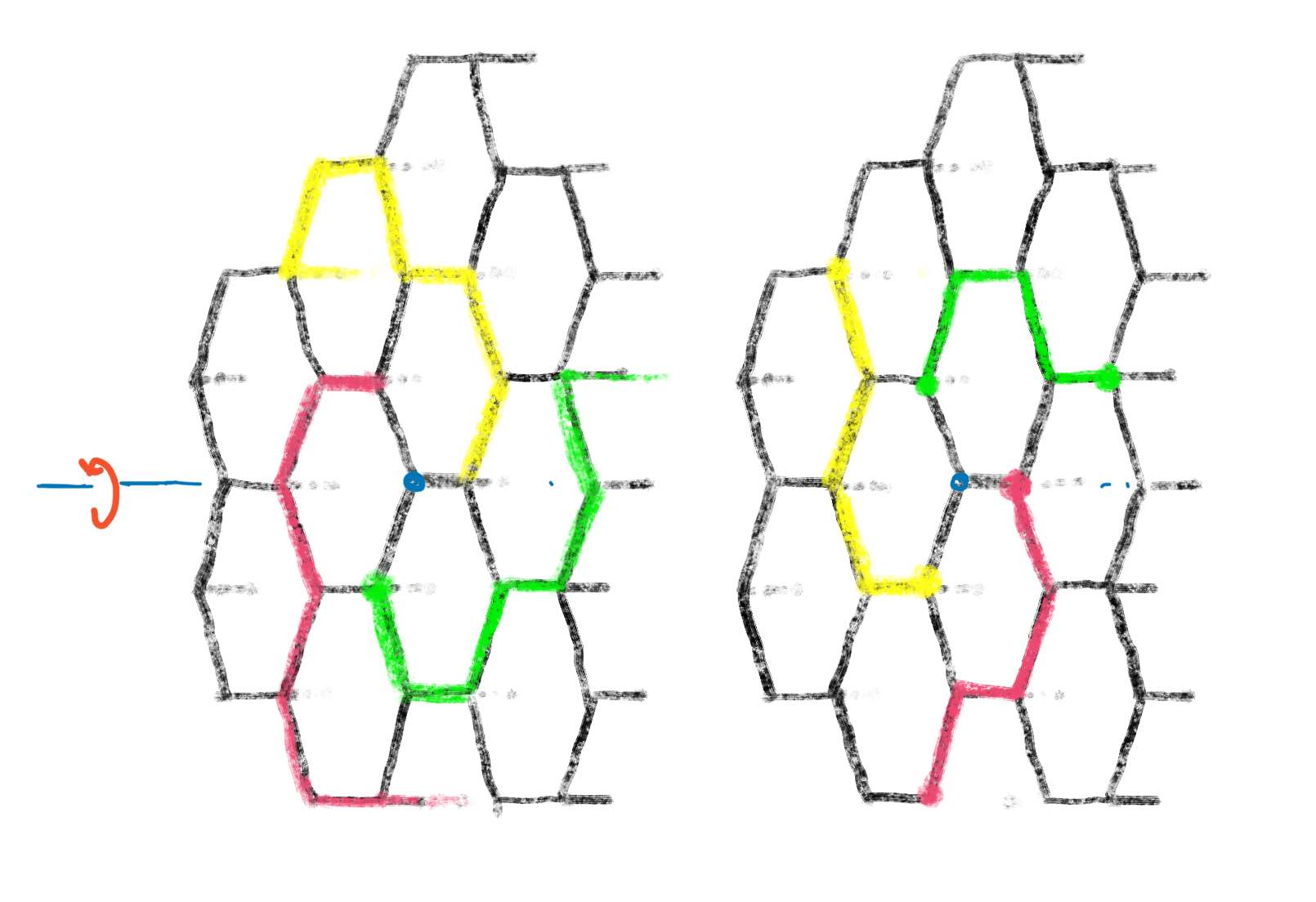}
    \captionof{figure}{A trefoil invariant under a threefold symmetry of the (sh) lattice}
\end{center}

The length of the piece is $12$, which justifies the inequality 
\begin{equation*}
    8\le l_3(\mbox{sh})\le 12.
\end{equation*}
we mentioned in the introduction. The lower bound is $8$ because a threefold symmetric trefoil must have at least length $24$. However, the stick number is $20$, so the minimums are not 3-periodic.

\subsection{Question 3: Is the trefoil always the minimum?}

Our final question is in the same spirit of the previous one and is based in the following observation: in all the cases known to the authors, the minimal knotting number is \textit{always} achieved by a trefoil. These cases are the overall knotting number, orthogonal stick number, cubic lattice, the honeycomb lattice (sh), the other lattices from \cite{mannminimalknottingnumbers} and the equilateral knots.
We suspect that is the case in $\tB$ as well.

Of course, this makes sense because the trefoil is the simplest nontrivial knot. This leads to our final question.

\begin{question}
    Is it true that for general conditions on the sticks used to construct knots, the minimal knotting number is achieved by trefoils? If so, is it exclusively by trefoils?
\end{question}

Somewhat equivalently, we could instead ask if it is possible to find a set of conditions on the sticks used to construct knots that, despite admitting trefoil knots to be constructed out of them, do not present trefoils as minima for the knotting number? If examples are to serve as our guide, such a construction doesn't seem possible. However, the amount of possibilities unexplored might hold such an example.


\bibliographystyle{acm}
\bibliography{Bibliography}

\end{document}